\documentclass[11pt]{amsart}
\usepackage[utf8]{inputenc}
\usepackage{amssymb,amsmath}
\usepackage{amsthm}
\usepackage{graphicx}
\usepackage{color}
\usepackage{tikz}
\usepackage{pgfplots}
\usepackage{enumitem}
\usepackage{setspace}
\usepackage{mathtools}
\usepackage{xfrac}

\usepackage{multicol} 
\usepackage{soul}
\usepackage{subcaption}
\usepackage{booktabs}
\usepackage{array}
\newcolumntype{L}[1]{>{\raggedright\let\newline\\\arraybackslash\hspace{0pt}}m{#1}}
\newcolumntype{C}[1]{>{\centering\let\newline\\\arraybackslash\hspace{0pt}}m{#1}}
\newcolumntype{R}[1]{>{\raggedleft\let\newline\\\arraybackslash\hspace{0pt}}m{#1}}
\usepackage{bbold}
\usepackage[linesnumbered, ruled,vlined]{algorithm2e}

\usepackage{relsize}
\usepackage{scalerel}
\usepackage{comment}

\usepackage[hidelinks]{hyperref}
\DeclareMathOperator*{\esssup}{ess\,sup}
\DeclareMathOperator*{\essinf}{ess\,inf}

\usepackage{bbold}

\newcommand{\supp}{\mathrm{supp}}
\newcommand{\id}{\mathbb{1}}
\newcommand{\Id}{\mathrm{Id}}

\newcommand{\opnorm}[1]{{\left\id\kern-0.25ex\left\id\kern-0.25ex\left\id #1 
		\right\id\kern-0.25ex\right\id\kern-0.25ex\right\id}}

\definecolor{myBlue}{RGB}{30,144,255}
\definecolor{myGreen}{RGB}{69,169,0}
\definecolor{myRed}{RGB}{165,12,42} 
\definecolor{myOrange}{RGB}{225,92,22} 
\definecolor{color2}{RGB}{255, 126, 126}
\definecolor{color3}{RGB}{0, 100, 0}
\definecolor{color1}{RGB}{176, 226, 255}

\tikzset{cross/.style={cross out, draw=black, minimum size=2*(#1-\pgflinewidth), inner sep=0pt, outer sep=0pt},
	cross/.default={0.6ex}}

\newtheorem{theorem}{Theorem}[section]

\newtheorem{lemma}[theorem]{Lemma}

\theoremstyle{definition}

\newtheorem{assumption}[theorem]{Assumption}
\newtheorem{remark}[theorem]{Remark}

\pgfplotstableset{
	every head row/.style={before row=\toprule, after row=\midrule},
	every last row/.style={after row=\bottomrule},
}

\numberwithin{theorem}{section}
\numberwithin{equation}{section}
\numberwithin{table}{section}
\numberwithin{figure}{section}
\allowdisplaybreaks
\newcommand{\low}[2]{\scalebox{#1}{$\SavedStyle#2$}}

\def\calT{\mathcal{T}}
\def\It{\mathcal{J}_\tau} 

\def\V{H^1_0({\Omega})}
\def\Vtr{V_\mathrm{tr}} %
\def\Vte{V_\mathrm{te}} %
\def\L{L^2({\Omega})}
\def\hVht{\hat{V}_{h,\tau}}
\def\VHT{\ThisStyle{{V}_{\low{.6}{H,\T}}}}
\def\hVHT{\ThisStyle{\hat{V}_{\low{.6}{H,\T}}}}
\def\hVht{\hat{V}_{h,\tau}}
\def\Vht{{V}_{h,\tau}}

\newcommand{\hVhtj}{\hat{V}_{h,\tau}^j}

\def\Rht{\mathcal{R}_{h,\tau}}

\def\VH{\ThisStyle{{V}_{\low{.6}{H}}}}
\def\Vh{{V}_h}
\def\Vt{{V}_\tau}
\def\VT{\ThisStyle{{V}_{\low{.6}{\T}}}}
\def\hVT{\ThisStyle{\hat{V}_{\low{.6}{\T}}}}
\def\hVt{\hat{V}_\tau}

\def\Wht{{W}_{h,\tau}}
\def\hWht{\hat{W}_{h,\tau}}
\def\WhtkK{\ThisStyle{{W}_{h,\tau,k}^{\low{.6}{K}}}}
\def\hWhtkK{\ThisStyle{\hat{W}_{h,\tau,k}^{\low{.6}{K}}}}
\def\hVhtkK{\ThisStyle{\hat{V}_{h,\tau,k}^{\low{.6}{K}}}}
\def\hVhtkKj{\ThisStyle{\hat{V}_{h,\tau,k}^{\low{.6}{K},j}}}

\def\VhtkK{\ThisStyle{{V}_{h,\tau,k}^{\low{.6}{K}}}}
\def\VHkK{\ThisStyle{{V}_{\low{.6}{H},k}^{\low{.6}{K}}}}
\def\VhkK{\ThisStyle{{V}_{h,k}^{\low{.6}{K}}}}
\def\N{\mathbb{N}}

\def\R{\mathbb{R}}

\def\IH{\ThisStyle{\mathfrak{I}_{\low{.6}{H}}}}
\def\IT{\ThisStyle{\hat{\mathfrak{I}}_{\low{.6}{\T}}}}
\def\IHk{\ThisStyle{I_{\low{.6}{H}}}} 
\def\IHT{\ThisStyle{\mathfrak{I}_{\low{.6}{H,\T}}}}
\def\hIHT{\ThisStyle{\hat{\mathfrak{I}}_{\low{.6}{H,\T}}}}

\def\QKi{\ThisStyle{\mathcal{Q}^{\low{.6}{K},i}}}

\def\QkKi{\ThisStyle{\mathcal{Q}_k^{\low{.6}{K},i}}}

\def\Qkl{\mathcal{Q}_{k,\ell}}

\def\QklKi{\ThisStyle{\mathcal{Q}_{k,\ell}^{\low{.6}{K},i}}}
\def\DKi{\ThisStyle{D_{\low{.6}{K},i}}}

\def\calQ{\mathcal{Q}}

\def\calB{\mathcal{B}}
\def\calN{\mathcal{N}}
\def\calK{\mathcal{K}}
\def\calC{\mathcal{C}}

\def\frakA{\mathfrak{A}}
\def\frakB{\mathfrak{B}}
\def\frakC{\mathfrak{C}}
\def\frakT{T} %

\def\Mhk{M_h} %
\def\MhkK{\ThisStyle{M_{h,\low{.6}{K}}}}
\def\MHk{\ThisStyle{M_{\low{.6}{H}}}} %
\def\Shki{S_{h}^{i}} %
\def\Shkim{S_{h}^{i-1}} %
\def\ShkiK{\ThisStyle{S_{h,\low{.6}{K}}^{i}}} %
\def\ShkimK{\ThisStyle{S_{h,\low{.6}{K}}^{i-1}}} %

\newcommand{\Nb}{\mathsf{N}}

\def\dt{\,\mathrm{d}t}
\def\ds{\,\mathrm{d}s}
\def\dx{\,\mathrm{d}x}

\def\intT{\int_0^\frakT}

\def\la{\langle}
\def\ra{\rangle}

\def\T{\ThisStyle{\low{.9}{\calT}}}
\def\NT{N_t}%
\def\NH{N_x}%
\def\Nt{n_t}%
\def\Nh{n_x}%
\def\calNH{\ThisStyle{\calN_{\low{.6}{H}}}}
\def\calTH{\ThisStyle{\calK_{\low{.6}{H}}}}
\def\calTh{\calK_h}

\def\piH{\ThisStyle{\pi_{\low{.6}{H}}}}
\def\PiH{\ThisStyle{\Pi_{\low{.6}{H}}}}

\def\tddt{\tfrac{\mathrm{d}}{\mathrm{d}t}}

\def\xf{15}
\def\yf{0}
\def\xs{12}
\def\ys{0}
\def\xt{7}
\def\yt{0}

\newcommand{\vH}{\ThisStyle{v_{\low{.6}{H}}}}
\newcommand{\wH}{\ThisStyle{w_{\low{.6}{H}}}}
\newcommand{\vh}{v_h}

\newcommand{\uh}{u_h}

\newcommand{\vht}{v_{h,\tau}}

\newcommand{\duht}{\dot{u}_{h,\tau}}
\newcommand{\wht}{w_{h,\tau}}
\newcommand{\uht}{u_{h,\tau}}

\newcommand{\tuHT}{\ThisStyle{\tilde{u}_{\low{.6}{H,\T}}}}

\newcommand{\uHT}{\ThisStyle{u_{\low{.6}{H,\T}}}}
\newcommand{\uHTkl}{\ThisStyle{u_{\low{.6}{H,\T},k,\ell}}}
\newcommand{\vHT}{\ThisStyle{v_{\low{.6}{H,\T}}}}
\newcommand{\wHT}{\ThisStyle{w_{\low{.6}{H,\T}}}}

\newcommand{\tuHTkl}{\ThisStyle{\tilde{u}_{\low{.6}{H,\T},k,\ell}}}
\newcommand{\xiK}{\ThisStyle{\xi_{j,k}^{\low{.6}{K},i}}}
\newcommand{\xiKp}{\ThisStyle{\xi_{j-1,k}^{\low{.6}{K},i}}} %
\newcommand{\xiKl}{\ThisStyle{\xi_{i+\ell-1,k}^{\low{.6}{K},i}}} %
\newcommand{\xiKf}{\ThisStyle{\xi^{\low{.6}{K},1}_{1,k}}} %

\newcommand{\lamKkf}{\ThisStyle{\lambda_{1,k}^{\low{0.6}{K},1}}}

\newcommand{\duHT}{\ThisStyle{\dot u_{\low{.6}{H,\T}}}}
\newcommand{\Qu}{\calQ\uHT}

\def\Cint{\ensuremath{C_\mathrm{int}} }
\def\tCint{\ensuremath{\tilde C_\mathrm{int}} }
\def\hCint{\ensuremath{\hat C_\mathrm{int}} }
\def\Cpc{\ensuremath{C_\mathrm{Pc}}}

\def\Cinv{\ensuremath{C_\mathrm{inv}}}
\def\Cst{\ensuremath{C_\mathrm{st}}}

\def\intOm{\int_{{\Omega}}}

\begin{document}
\title[A Space-Time Multiscale Method for Parabolic Problems]{A Space-Time Multiscale Method for Parabolic Problems}
\author[P.~Ljung, R.~Maier, A.~M\aa{}lqvist]{Per~Ljung$^\dagger$, Roland~Maier$^\dagger$, Axel~M\aa{}lqvist$^\dagger$}
\address{${}^{\dagger}$ Department of Mathematical Sciences, Chalmers University of Technology and University of Gothenburg, 412 96 G\"oteborg, Sweden}
\email{perlj@chalmers.se, roland.maier@chalmers.se, axel@chalmers.se}
\date{\today}
\keywords{}
\begin{abstract}
We present a space-time multiscale method for a parabolic model problem with an underlying coefficient that may be highly oscillatory with respect to both the spatial and the temporal variables. The method is based on the framework of the Variational Multiscale Method in the context of a space-time formulation and computes a coarse-scale representation of the differential operator that is enriched by auxiliary space-time corrector functions. Once computed, the coarse-scale representation allows us to efficiently obtain well-approximating discrete solutions for multiple right-hand sides. We prove first-order convergence independently of the oscillation scales in the coefficient and illustrate how the space-time correctors decay exponentially in both space and time, making it possible to localize the corresponding computations. This localization allows us to define a practical and computationally efficient method in terms of complexity and memory, for which we provide a posteriori error estimates and present numerical examples. 
\end{abstract}
\maketitle
{\tiny {\bf Keywords.} multiscale method, numerical homogenization, space-time method, parabolic problem
}\\
\indent
{\tiny {\bf AMS subject classifications.}  
{\bf 65M12}, {\bf 65M60} 
} 

\section{Introduction}\label{s:intro}

In this {paper}, we study numerical solutions to a parabolic equation of the form
\begin{equation*}
\dot{u} - \nabla \cdot (A\nabla u) = f,
\end{equation*}
on a given bounded domain and with a coefficient $A$ that varies rapidly in both space and time. Such an equation arises in several applications, such as heat transfer, transport of groundwater, and the modeling of pressure in compressible flows. 
It is well-known that classical finite element methods only lead to reasonable approximations if the variations of $A$ are resolved by the underlying discretization.  
If the coefficient is highly oscillatory in space and time, these conditions are quite severe and computations quickly become unfeasible. 	

Multiscale methods generally aim to circumvent this problem and construct appropriate coarse-scale spaces with improved approximation properties compared to classical finite element methods by including problem-dependent fine-scale information. 
Some of these approaches emerged from analytical homogenization theory, such as the Heterogeneous Multiscale Method (HMM)~\cite{EE03,AbdEEV12} or the Multiscale Finite Element Method~\cite{HowW97}. The analysis of these methods is based on analytical homogenization and therefore requires assumptions on the coefficient such as scale separation or periodicity. 
To overcome these restrictions {on a theoretical level}, so-called \emph{numerical homogenization} methods were developed. For the elliptic setting, prominent examples include generalized (Multiscale) Finite Element Methods~\cite{BabO83,BabL11,EfeGH13}, the Localized Orthogonal Decomposition method (LOD)~\cite{MalP14,HenP13}, Rough Polyharmonic Splines~\cite{OwhZB14}, or gamblets~\cite{Owh17}. For an overview on numerical homogenization techniques, we also refer to~\cite{OwhS19,MalP20,AltHP21}.
In this contribution, we propose a space-time multiscale method based on the framework of the Variational Multiscale Method~\cite{Hug95,HugS96,HugFMQ98} and the above-mentioned LOD method that is an extension of the Variational Multiscale Method with a particular focus on appropriate localization. Our construction further relies on a space-time Petrov--Galerkin discretization as, e.g., considered in~\cite{UrbP14,Ste15} and treats highly varying coefficients in both space and time.

Note that the LOD has previously been employed for various time-dependent problems, including parabolic-type problems \cite{MalP17, MalP18, AltCMPP18} and wave-type equations \cite{AbdH17,PetS17,MaiP19,LjuMP20,GeeM21}.
However, all these works treat underlying coefficients that only depend on spatial variables and are independent of time, such that the corresponding LOD methods fail to take rapid temporal variations into account. 

In the context of space- and time-varying coefficients for parabolic problems, the approach in~\cite{OwhZ0708} deals with multiscale effects by an appropriate coordinate transformation. This however, requires the solution of global fine-scale problems. More recently, locally defined space-time multiscale methods for parabolic equations have been considered, for instance, in~\cite{ChuELY18,SchS20}, and in~\cite{HuLCEP21} in connection with moving channels. 

The main goal of this paper is to start from the basic ideas of the LOD and introduce and analyze a multiscale method that deals with arbitrary space- and time-dependent coefficients and still allows for the favorable localization properties of the LOD.
To this end, we construct a Petrov--Galerkin method that is based on a decomposition of the solution space into a coarse approximation space and the remainder space. The general idea is to enrich the coarse space by appropriate functions from the remainder space. This is done by the solution of so-called corrector problems that are themselves parabolic-type problems. They are totally independent of each other in both spatial and temporal sense and therefore allow for parallel computations. 
The approach results in a coarse-scale representation of the differential operator with the great advantage that it can be efficiently applied to multiple source functions on the right-hand side. 
We also present and analyze a localized version of the method and illustrate its practical feasibility. 

The remaining parts of the paper are structured as follows. In Section~\ref{s:modelproblem}, we introduce the weak formulation and a classical space-time finite element discretization for a parabolic model problem, along with necessary assumptions. Section~\ref{s:msmethod} is devoted to the construction of the space-time multiscale method and the error analysis for the ideal method. We then motivate and introduce a localized variant of the method and prove a posteriori localization estimates (Section~\ref{s:locmethod}). Finally, we discuss details regarding implementation (Section~\ref{s:implementation}) and provide numerical examples in Section~\ref{s:numericalexamples} before we conclude in Section~\ref{s:conclusion}.

\subsection*{Notation} In the following, $C$ denotes a generic constant that is independent of the mesh sizes $h,H$ and the time steps $\tau,\T$. $C$ might change from line to line in the estimates. Further, we write $a \lesssim b$ if there exists a constant $C$ such that $a \leq C\,b$.

\section{Model Problem and Discretization}
\label{s:modelproblem}

We consider a parabolic equation of the form
\begin{subequations}\label{eq:modelprob}
\begin{alignat}{2}
\dot{u} - \nabla \cdot (A\nabla u) &= f, \quad &&\text{in ${\Omega} \times (0,\frakT]$}, \label{eq:para}\\
u&= 0, \quad &&\text{on $\Gamma \times (0,\frakT]$}, \label{eq:dbc}\\
u(0) &= 0, \quad &&\text{in ${\Omega}$} \label{eq:init1}
\end{alignat}
\end{subequations}
on a polygonal (or polyhedral) domain ${\Omega} \subset\mathbb{R}^d, \ d=2,3,$ with boundary $\Gamma := \partial {\Omega}$. Further, $\frakT > 0$ denotes the final time. The coefficient $A$ describes a (possibly highly oscillatory) diffusion and $f \in L^2(H^{-1})$ denotes the source function of the system. We have in mind a coefficient $A\colon {\Omega}\times(0,\frakT) \to \mathbb{R}^{d\times d}$ that depends on both the spatial and temporal variables. We emphasize that the choice of zero initial data is to simplify the presentation and nonzero initial data can be considered as well with just a few alterations throughout this work. 

\subsection{Weak formulation}
Let $H^1_0({\Omega})$ be the classical Sobolev space with zero trace on $\Gamma$ and the norm 
\begin{equation*}
\|v\|_{\V} =  \|\nabla v\|_{L^2({\Omega})},\quad v \in \V.
\end{equation*}
Its dual space is denoted with $H^{-1}({\Omega})$.  Moreover, let $L^2(0,\frakT; \mathcal{B})$ 
be the standard Bochner space  
with norm 
\begin{align*}
\|v\|_{L^2(0,\frakT; \mathcal{B})} &= \bigg( \int_0^\frakT \|v\|_{\mathcal{B}}^2\dt \bigg)^{1/2},
\end{align*}
where $\calB$ is a Banach space with norm $\|\cdot\|_\calB$. Throughout this paper, we frequently abbreviate occurring Bochner spaces without the interval and the domain and write, e.g., $L^2(H^1_0) := L^2(0,\frakT;H^1_0({\Omega}))$. Further, the following assumptions are made on the data.

\begin{assumption}[Regularity]\label{ass:reg}
The diffusion coefficient $A\in L^\infty({\Omega}\times (0,\frakT) ; \mathbb{R}^{d\times d}_\mathrm{{sym}})$ satisfies
\begin{align*}
0 < \alpha := \essinf_{(x,t)\in {\Omega}\times(0,\frakT)} \inf_{v\in \mathbb{R}^d\backslash \{0\}} \frac{A(x,t)v\cdot v}{v\cdot v} \leq \esssup_{(x,t)\in {\Omega}\times(0,\frakT)} \sup_{v\in \mathbb{R}^d\backslash \{0\}} \frac{A(x,t)v\cdot v}{v\cdot v} =: \beta < \infty,
\end{align*}
and we set $\bar{\beta} := \|A\|_{L^\infty(\Omega\times (0,T))}$. 
\end{assumption}

In order to discretize~\eqref{eq:modelprob}, we first express the model in weak form and consider the following space-time formulation: find $u \in L^2(H^1_0)\cap H^1(H^{-1})$ with $u(0) = 0$ such that
\begin{equation}\label{eq:exsol}
\intT \la\dot u, v\ra + a(t;u,v) \dt = \intT \la f,v\ra \dt
\end{equation}
for all $v \in L^2(H^1_0)$. Hereby, $\la\cdot,\cdot\ra$ denotes the duality pairing of $H^{-1}({\Omega})$ and $H^1_0({\Omega})$, which reduces to the $L^2$-inner product if the function in the first component additionally is in $L^2({\Omega})$. Further, $a(t;\cdot,\cdot)\colon\V \times \V \to \R$ denotes the bilinear form defined by
\begin{equation*}
a(t;v,w) := \intOm A(\cdot,t) \nabla v \cdot \nabla w \dx
\end{equation*}
for almost all $t \in (0,\frakT)$. In the following, we will drop the explicit $t$-dependence in the bilinear form $a(t;\cdot,\cdot)$ and simply write $a(\cdot,\cdot)$ when $a$ is considered in connection with an integral in time.
If Assumption~\ref{ass:reg} holds, problem \eqref{eq:exsol} has a unique solution, see, e.g.,~\cite{Ste15}.

Next, we introduce a discretized setting based on the weak form~\eqref{eq:exsol} by means of a classical finite element method. Therefore, we introduce finite element spaces for the spatial and temporal domains separately and then combine them to particular space-time finite element spaces. We then define the corresponding Petrov--Galerkin discretization and finish this section by discussing problems that occur with the standard approach in the setting where the oscillations in the coefficient are not resolved.

\subsection{A classical space-time finite element method}
\label{ss:FE}

Let $\{\calTh\}_{h>0}$ be a family of regular partitions (or meshes) of the domain ${\Omega}$ into simplices in the sense of~\cite[Ch.~2~\&~3]{Cia78}. Further, we assume shape regularity and quasi-uniformity of the meshes. For an element $K\in \calTh$, we let $h_K := \text{diam}(K)$ and further define the mesh size as the largest diameter of the partition, i.e., $h:=\max_{K\in \calTh} h_K$. We construct a classical finite element space using continuous piecewise affine functions by defining
\begin{align*}
\Vh := \{v \in \V \,\colon\, v\id_{K},\,K\in\calTh,\text{ is a polynomial of degree }\leq 1\},
\end{align*}
where $\id_S$ is the indicator function of some set $S$. 
Let $\mathcal{N}_h$ be the set of interior nodes of $V_h$ and denote with $\{\varphi_x\}_{x\in\mathcal{N}_h}$ the set of nodal basis functions corresponding to the mesh $\calTh$. 

For the temporal discretization, let $\tau > 0$ be some time step and set $t_i = i \tau$, $i = 0,\ldots,\Nt$. For simplicity, we assume that $t_{\Nt} = \frakT$. We denote with $\It$ the partition of the time interval $[0,\frakT]$ into sub-intervals $[t_{i-1},t_i]$, $i = 1,\ldots,\Nt$. Based on this partition, we introduce two discrete spaces $\hVt$ and $\Vt$ for the temporal discretization of the trial and test space, respectively. Precisely, we have that 
\begin{equation*}
\hVt := \{v \in H^1(0,\frakT)\,\colon\, v\id_{I},\,I \in \It,\text{ is a polynomial of degree }\leq 1 \text{ and } v(0) = 0 \}
\end{equation*}
with nodal basis $\{\zeta_i\}_{i=1}^{\Nt}$ with respect to $\It$, and
\begin{equation*}
\Vt := \{v \in L^2(0,\frakT)\,\colon\, v\id_{I},\,I \in \It,\text{ is constant} \}
\end{equation*}
with piecewise constant local basis functions $\{\chi_i\}_{i=1}^{\Nt}$, where $\chi_i = \id_{[t_{i-1},t_{i}]}$. 

At this point, we restrict the possible choices of coefficients $A$. Note that the following assumption is not severe and aims at avoiding technicalities in the later parts of this paper.
\begin{assumption}[Structure of $A$]\label{ass:strucA}
We assume that for almost all $x \in {\Omega}$, $A(x,\cdot)$ is piecewise constant on the (fine) intervals $I \in \It$.
\end{assumption}

Based on the above definitions of spatial and temporal finite element spaces, we now introduce the corresponding tensor-product space-time finite element spaces with respect to the domain ${\Omega} \times (0,\frakT)$. We set
\begin{equation}\label{eq:stspaces}
\hVht : = \Vh \times \hVt, \qquad \Vht := \Vh \times \Vt
\end{equation}
for the trial and test space, respectively. Note that the basis functions that span the trial space $\hVht$ are $\{\varphi_x\zeta_i\}_{x\,\in\,\mathcal{N}_h,\,i=1,\ldots,\Nt}$, so-called \emph{pyramids} in space-time, while for the test space we have the basis $\{\varphi_x\chi_i\}_{x\,\in\,\mathcal{N}_h,\, i=1,\ldots,\Nt}$, so-called \emph{tents} in space-time. 

With the discrete space-time finite element spaces~\eqref{eq:stspaces}, we now seek a discrete solution $u_{h,\tau} \in \hVht$ such that
\begin{equation}\label{eq:discretizedsol}
\intT \la\duht, \vht\ra + a(\uht,\vht) \dt = \intT \la f,\vht\ra \dt
\end{equation}
for all $\vht \in \Vht$. Note that, with the explicit choices of the temporal discretization, we can reformulate~\eqref{eq:discretizedsol} in terms of a classical time-stepping scheme. To this end, we decompose the trial function as a sum of its temporal basis functions, i.e.,
\begin{align*}
\uht = \sum_{j=1}^{\Nt} u_h^j\, \zeta_j
\end{align*}
with $u_h^j \in \Vh$ and use test functions of the form $\vht = \vh \chi_i$ with $\vh \in \Vh$. Inserting these expressions into \eqref{eq:discretizedsol}, we obtain
\begin{align}
\int_{t_{i-1}}^{t_i} \la \uh^{i-1}\dot{\zeta}_{i-1},\vh \ra + \la \uh^i\,\dot{\zeta}_{i}, \vh\ra + a(\uh^{i-1}\,\zeta_{i-1},\vh) + a(\uh^i\,\zeta_i, \vh) \dt = \int_{t_{i-1}}^{t_i} \la f,\vh\ra \dt 
\end{align}
for all $\vh \in \Vh$ and for $i=1,2,\ldots,\Nt$, where $\zeta_{i-1}$ and $\zeta_i$ remain due to their support on $[t_{i-1}, t_i]$. By further evaluating the integrals, we obtain the well-known Crank--Nicolson scheme: for $i=1,2,\ldots,\Nt$, find $\uh^i\in \Vh$ such that
\begin{equation*}\label{eq:CNscheme}
\la \uh^i, \vh\ra + \frac{\tau}{2}\,a(t_{i-1/2};\uh^i,\vh) = \int_{t_{i-1}}^{t_i} \la f,\vh\ra \dt + \la\uh^{i-1},\vh\ra - \frac{\tau}{2}\,a(t_{i-1/2};\uh^{i-1},\vh)
\end{equation*}
for all $\vh \in \Vh$, where we set $t_{i-1/2} := (i-1/2)\tau$ and explicitly use Assumption~\ref{ass:strucA}.

It is well-known that the Crank--Nicolson scheme is unconditionally stable, see, e.g., \cite[Ch.~7]{KnaA03}. Nevertheless, within this work, we rely on the space-time formulation~\eqref{eq:discretizedsol} and aim for a stability estimate with respect to the spaces $\Vtr:=L^2(H^1_0)\cap H^1(H^{-1})$ and $\Vte:=L^2(H^1_0)$. 
The well-posedness directly follows from the inf-sup estimate for the bilinear form $\frakA\colon\Vtr\times\Vte\to\R$ defined by
\begin{equation}\label{eq:defA}
\frakA(v,w) := \intT \la\dot v, w\ra + a(v,w) \dt
\end{equation}
that is stated in the following lemma. First, however, we define appropriate norms on $\Vtr$ and $\Vte$, respectively, that read 
\begin{equation*}
\|v\|^2_{\Vtr} := \intT  \|\dot v(t,\cdot)\|^2_{H^{-1}({\Omega})} + \|\nabla \overline{v}(t,\cdot)\|_{\L}^2 \dt,\qquad
\|v\|^2_{\Vte} := \intT \|\nabla v(t,\cdot)\|_{\L}^2 \dt,
\end{equation*}
where $\overline{v}:=\sum_{i = 1}^{\Nt} \big(\tau^{-1}\int_{t_{i-1}}^{t_i}v(s,\cdot) \ds\big)\,\chi_i$ is the mean with respect to the temporal discretization. 
We emphasize that these are norms due to the Friedrichs inequality. In the following, the mean $\overline{(\cdot)}$ is always taken with respect to the (fine) time step $\tau$.

\begin{lemma}[Inf-sup condition]\label{lem:infsup}
The bilinear form $\frakA$ fulfills 
\begin{equation*}
\adjustlimits\inf_{\vht \in \hVht}\sup_{\wht \in \Vht} \frac{\frakA(\vht,\wht)}{\|\vht\|_{\Vtr}\,\|\wht\|_{\Vte}} \geq c_\frakA 
\end{equation*}
with $c_\frakA = \min\{\alpha,\sqrt{\alpha/\beta}\}$, where here and in the following we implicitly exclude zero in the infimum and supremum. 
\end{lemma}

\begin{proof}
The lemma follows directly from~\cite[Prop.~2.9]{UrbP14} when the $H^1_0$-norm $\|\nabla\cdot\|_{L^2(\Omega)}$ is used instead of $\|A^{1/2}\nabla \cdot\|_{L^2(\Omega)}$ and the last term in the definition of the norm in $\Vtr$ is omitted. 
\end{proof}
From Lemma~\ref{lem:infsup}, we directly obtain the well-posedness of~\eqref{eq:discretizedsol} with stability estimate
\begin{equation*}
\|\uht\|_{\Vtr} \leq c_\frakA^{-1}\,\|f\|_{L^2(H^{-1})}, 
\end{equation*}
see, e.g., \cite{Bab71}. 
If the solution $u$ of~\eqref{eq:exsol} is sufficiently regular, one can, for instance, show an error estimate of the form
\begin{equation}\label{eq:erruht}
\|u - \uht\|_{\Vte} \leq C\,(h^{s}+\tau^s)\big(\|u\|_{L^2(0,\frakT;H^{1+s}({\Omega}))} + \|u\|_{H^{1+s}(0,\frakT;L^2({\Omega}))}\big)
\end{equation}
with $s \in (0,1]$, see, e.g.,~\cite{Hac81,Zan20}. 
However, in the presence of an oscillatory coefficient $A$ on a scale $\varepsilon$, the right-hand side of~\eqref{eq:erruht} behaves like $((h+\tau)/\varepsilon)^s$ due to the scaling in the norms. That is, the error estimate is only reliable in the regime $h + \tau < \varepsilon$ and otherwise leads to unsatisfactory results. 

Therefore, the space $\hVht$ is not suited for the approximation of~\eqref{eq:exsol} in an under-resolved setting with $h + \tau \geq \varepsilon$, which is a well-known issue with varying coefficients. A~possible solution for this difficulty is provided in the following section. 
	
\section{A Space-Time Multiscale Method}
\label{s:msmethod}

The goal of this section is to establish a multiscale method that efficiently solves~\eqref{eq:exsol} independently of rapid space-time variations in the diffusion coefficient. To this end, we first introduce some notation for the discretization. Let $\VH$ be a finite element space defined analogously to $V_h$ in Section \ref{ss:FE}, but with larger mesh size $H>h$. Moreover, we assume that the mesh $\calTh$ is a refinement of $\calTH$ such that $\VH \subset V_h$. The set of interior nodes of $\VH$ is given by $\calNH$. We will commonly refer to $\VH$ as the coarse-scale space and to $V_h$ as the fine-scale space. 

For the temporal discretization, let $\T>\tau>0$ be some coarse time step and set $T_i = i\,\T$, $i=0,1,\ldots,\NT$, where for simplicity we again assume that $T_{\NT} = \frakT$. Similar to Section~\ref{ss:FE}, we let $\mathcal{J}_{\scalebox{.6}{$\T$}}$ denote the coarse partition of the time domain $[0,\frakT]$ into sub-intervals $[T_{i-1},T_i]$, $i=1,2,\ldots,\NT$. Given this partition, we define the coarse trial and test spaces $\hVT$ and $\VT$ in complete analogy with $\hVt$ and $\Vt$, based on piecewise linear and piecewise constant basis functions, respectively. 
As for the spatial discretization, we assume that the coarse and fine time steps $\T$ and $\tau$ are such that $\T/\tau \in \N$. 
The coarse tensor-product space-time finite element trial and test space with respect to ${\Omega} \times (0,\frakT)$ on the coarse scale are given by
\begin{align*}
\hVHT := \VH \times \hVT, \qquad \VHT := \VH \times \VT.
\end{align*}
Let us mention that with regard to possible fine oscillations in the diffusion coefficient $A$, we assume that $h$ and $\tau$ (as well as the respective spaces) are fine enough to resolve these fine quantities, while $H$ and $\T$ are parameters on an under-resolved coarse scale.

We emphasize that we reuse the notation from Section~\ref{s:modelproblem} in the following and write $\{\varphi_x\zeta_i\}_{x\,\in\,\mathcal{N}_H,\, i=1,\ldots,\NT}$ (pyramids) and $\{\varphi_x\chi_i\}_{x\,\in\,\mathcal{N}_H,\, i=1,\ldots,\NT }$ (tents) for the canonical basis functions of the (coarse) spaces $\hVHT$ and $\VHT$, respectively.

In the following subsections, we derive a method based on a clever decomposition of the fine spaces $\hVht$ and $\Vht$ each into a coarse and a fine subspace. These decompositions allow us to construct an alternative and more sophisticated decomposition for the trial space, which incorporates suitable fine quantities into the coarse space.

\subsection{Quasi-interpolation and space decompositions}

A straight-forward decomposition consists in separating the trial space $\hVht$ into the coarse space $\hVHT$ and a fine-scale remainder space. Therefore, let $\IH\colon \L \rightarrow \VH$ be a \emph{projective quasi-interpolation operator} that maps into the finite element space $\VH$. That is, $\IH$ fulfills
\begin{subequations}\label{eq:IHloc}
\begin{alignat}{2}
H^{-1}\,\|(\Id - \IH)v\|_{L^2(K)} + \|\nabla \IH v \|_{L^2(K)} &\leq \tCint\, \|\nabla v\|_{L^2(\Nb(K))}, \quad &&v \in \V,\label{eq:IH1loc}\\ 
\|\IH v\|_{L^2(K)} &\leq \tCint\, \|v\|_{L^2(\Nb(K))}, \quad &&v \in \L, \label{eq:IH2loc}\\
\IH \circ \IH &= \IH, \label{eq:IH3}
\end{alignat}
\end{subequations}
where $\Nb(K)$ is the neighborhood of the element $K \in \calTH$. To be more precise we have, for any subset $S \subseteq {\Omega}$ (e.g., an element or a node),
\begin{equation}
\Nb(S) := \bigcup\bigl\{\overline{K}\in\calTH\,\colon\,\overline{K}\cap\overline{S} \neq\emptyset\bigr\}.
\label{eq:NT}
\end{equation}
With \eqref{eq:IHloc}, we directly obtain a global estimate of the form
\begin{subequations}\label{eq:IH}
\begin{alignat}{2}
H^{-1}\,\|(\Id - \IH)v\|_{\L} + \|\nabla \IH v \|_{\L} &\leq \Cint\, \|\nabla v\|_{\L}, \quad &&v \in \V,\label{eq:IH1}\\ 
\|\IH v\|_{\L} &\leq \Cint\, \|v\|_{\L}, \quad &&v \in \L, \label{eq:IH2}
\end{alignat}
\end{subequations}
where the constant \Cint differs from \tCint by a moderate multiplicative factor. 
Although an explicit characterization of the operator $\IH$ is not essential for the definition of the method below, we restrict ourselves to the particular choice $\IH := \piH \circ \PiH$, where $\PiH$ denotes the piecewise $L^2$-projection onto $P_1(\calTH)$, the space of piecewise affine functions. The operator $\piH$, for any $\vH\in P_1(\calTH)$ and any vertex $z$ of $\calTH$, performs an averaging in the sense that
\begin{equation*}
\big(\piH(\vH)\big)(z) := \sum_{K\in\calTH:\atop z \in \overline{K}} \big({\vH\id}_K\big)(z)\cdot\frac{1}{\mathrm{card}\{K' \in \calTH\,\colon\, z \in \overline{K'}\}},
\end{equation*}
i.e., it takes in each node the average of the values in the adjacent elements. 
We refer to~\cite{Osw93,Bre94,ErnG17} for a proof of \eqref{eq:IH1loc}--\eqref{eq:IH3} for this particular choice.

Next, we combine this spatial operator with pointwise interpolation in time for the trial space and with interval-wise averaging for the test space. This leads to the definition of the following fine-scale remainder spaces,
\begin{align*}
\hWht &:= \Big\{w \in \hVht\,\colon\, \IH w(\cdot,T_i) = 0 \text{ for all } i = 1,\ldots,\NT\Big\},\\
\Wht &:= \Big\{w \in \Vht\,\colon\, \T^{-1}\int_{T_{i-1}}^{T_i}\IH w \dt = 0 \text{ for all } i = 1,\ldots,\NT\Big\}.
\end{align*}
Note that, by construction, $\hVht= \hVHT \oplus \hWht$ and $\Vht= \VHT \oplus \Wht$. As already discussed in Section~\ref{ss:FE}, the space $\hVHT$ lacks suitable approximation properties if the coarse discretization parameters do not resolve variations in the coefficient. Therefore, we modify the decomposition $\hVht= \hVHT \oplus \hWht$ and add certain fine information to the space $\hVHT$ without changing the overall number of degrees of freedom. A corresponding (ideal) multiscale method is introduced in the next subsection. First, however, we require interpolation operators on the coarse spaces. We define $\hIHT\colon\hVht \to \hVHT$ and $\IHT\colon\Vht \to \VHT$ by
\begin{align*}
\hIHT v := \IT\IH v = \sum_{i = 1}^{\NT} \IH v(\cdot,t_i)\, \zeta_i, \qquad 
\IHT v := \sum_{i = 1}^{\NT} \Big(\T^{-1}\int_{T_{i-1}}^{T_i} \IH v \dt\Big)\, \chi_i,
\end{align*}
where $\IT$ is the nodal interpolation with respect to $\mathcal{J}_{\scalebox{.6}{$\T$}}$.
Note that these operators are constructed in a way such that they do not spread information in time. In spatial sense, however, the nature of $\IH$ may increase the support of a given function by at most one layer of elements. 
We conclude this subsection by presenting useful stability estimates for the operators $\hIHT$ and $\IHT$.
\begin{lemma}[Stable interpolation]\label{lem:boundInterpolation}
It holds that
\begin{equation}\label{eq:stabIHT}
\|\IHT v\|_{\Vte} \leq \Cint\,\|v\|_{\Vte}
\end{equation}
for all $v \in \Vte$ with $\Cint$ from~\eqref{eq:IH}. 
If $\T \leq C H$, there exists a constant $\hCint$ such that
\begin{equation}\label{eq:stabhIHT}
\|\hIHT v\|_{\Vtr} \leq \hCint H^{-1}\,\|v\|_{\Vtr}
\end{equation}
for $v \in \Vtr$.
\end{lemma}
\begin{proof}
For any $v \in \Vte$ and $i = 1,\ldots,\NT$, we have
\begin{align*}
\|\IHT v \id_{[T_{i-1},T_{i}]}\|^2_{\Vte} & = \int_{T_{i-1}}^{T_{i}} \|\nabla\IHT v\|^2_{\L} \dt = \int_{T_{i-1}}^{T_{i}}  \int_{{\Omega}} \big|\T^{-1}\int_{T_{i-1}}^{T_{i}} \nabla \IH v\big|^2\ds\dt \nonumber\\
& \leq \int_{T_{i-1}}^{T_{i}} \int_\Omega \T^{-1}\int_{T_{i-1}}^{T_{i}} |\nabla \IH v\big|^2 \ds \dt \label{eq:boundIHT}
\leq \Cint^2\, \|v \id_{[T_{i-1},T_{i}]}\|^2_{\Vte} \nonumber
\end{align*}
using Jensen's inequality and~\eqref{eq:IH1}. Globally, we thus have
\begin{equation*}
\|\IHT v\|_{\Vte} \leq \Cint\,\|v\|_{\Vte}.
\end{equation*}
For the proof of~\eqref{eq:stabhIHT}, recall that $\IH = \piH\circ\PiH$. Further, we note that 
\begin{equation}\label{eq:estDual}
\begin{aligned}
\|\IH v\|_{H^{-1}(\Omega)} &\leq \|\IH v\|_{L^2(\Omega)} \leq C\,\|\PiH v\|_{L^2(\Omega)} = C \sup_{w \in L^2(\Omega)}\frac{\langle \PiH v,w\rangle}{\|w\|_{L^2(\Omega)}} 
\\&\leq C\sup_{\wH \in \VH}\frac{\langle v,\wH\rangle}{\|\wH\|_{L^2(\Omega)}} \leq C \sup_{\wH \in \VH}\frac{\Cinv\,\langle v,\wH\rangle}{H\,\|\nabla\wH\|_{L^2(\Omega)}} \leq \Cst H^{-1}\|v\|_{H^{-1}(\Omega)}
\end{aligned}
\end{equation}
using the stability of $\piH$ (cf.~\cite[Lem.~6.1]{ErnG17}), the stability of the projection $\PiH$, and the inverse inequality
\begin{equation}\label{eq:invIneq}
\|\nabla \vH\|_{\L} \leq \Cinv H^{-1} \|\vH\|_{\L}, \quad \vH \in \VH.
\end{equation}
Let now $v \in \Vtr$. With the $H^1$-stability of the nodal interpolation in time~\cite[Thm.~1]{Dic13}, the $L^2$-stability of the local $L^2$-projection $\overline{(\cdot)}$, and the Friedrichs inequality on each coarse interval, we obtain
\begin{align*}
\|\hIHT v\|_{\Vtr}^2 & = \intT \big\|\tddt\hIHT v \big\|_{H^{-1}(\Omega)}^2 \dt + \intT \intOm\big|\overline{\nabla \hIHT v}\big|^2 \dx\dt \nonumber\\
& \leq \intT \big\|\tddt\IH v \big\|_{H^{-1}(\Omega)}^2 \dt + 2\intT \intOm\big|\overline{\nabla \IH v}\big|^2 \dx\dt  \\&\hspace{7.2cm} + 2\intT \intOm\big|\overline{(\Id - \IT)\nabla \IH v}\big|^2 \dx\dt 
\\
& \leq \intT \big\|\tddt\IH v \big\|_{H^{-1}(\Omega)}^2 \dt + 2\intT \big\|\nabla \IH \overline{v}\big\|_{L^2(\Omega)}^2 \dt + C\,\T^2\intT \big\|\nabla \IH \tddt v\big\|_{L^2(\Omega)}^2 \dt. \nonumber
\end{align*}
Utilizing the stability of $\IH$ (cf.~\eqref{eq:IH}), \eqref{eq:estDual}, and the inverse inequality~\eqref{eq:invIneq}, we further estimate
\begin{align*}
\intT \big\|\tddt\IH v &\big\|_{H^{-1}(\Omega)}^2 \dt + 2\intT \big\|\nabla \IH \overline{v}\big\|_{L^2(\Omega)}^2 \dt + C\,\T^2\intT \big\|\nabla \IH \tddt v\big\|_{L^2(\Omega)}^2 \dt \nonumber\\
& \leq \Cst^2H^{-2}\intT \big\|\tddt v \big\|_{H^{-1}(\Omega)}^2 \dt + 2\,\Cint^2\intT \big\|\nabla \overline{v}\big\|_{L^2(\Omega)}^2 \dt 
\\&\qquad\qquad\qquad\qquad\qquad\qquad+ C\Cinv^2\Cst^2\Big(\frac{\T^2}{H^4}\Big)\intT \big\|\tddt v\big\|_{H^{-1}(\Omega)}^2 \dt. \nonumber
\end{align*}
With $\T \leq CH$, we finally get
\begin{equation*}
\|\hIHT v\|^2_{\Vtr} \leq \hCint^2H^{-2}\,\|v\|^2_{\Vtr}. \qedhere
\end{equation*}
\end{proof}

\subsection{Ideal method}

In this subsection, we introduce a space-time multiscale method based on a correction of coarse functions by appropriate fine-scale functions.   
The proposed method is based on the Variational Multiscale Method, which was developed in~\cite{Hug95,HugFMQ98}, {and incorporates ideas from the LOD method introduced in~\cite{MalP14}}. The main idea is to decompose the solution into a coarse part in $\hVHT$ and a remainder part in $\hWht$  and then consider~\eqref{eq:exsol} for test functions in the coarse test space $\VHT$ and the fine test space $\Wht$ separately. Consequently, we arrive at two equations, one on the coarse scale and one on the fine scale. The coupling of these two equations can be seen as a \emph{correction} of the coarse-scale part by appropriate fine functions, whose main purpose is to include the fine-scale space-time behavior of the diffusion coefficient into the coarse-scale problem. This leads to a coarse-scale space with improved approximation properties.

We phrase the method~\eqref{eq:VMM} below in the classical framework of the Variational Multiscale Method. Note that the second equation~\eqref{eq:VMMfine} defines the \emph{correction operator} $\calQ\colon\hVht\to\hWht$ that is employed in the first equation~\eqref{eq:VMMcoarse}. Altogether, the (ideal) multiscale method reads: find $\tuHT = \uHT + \calQ \uHT \in (\Id+\calQ)\hVHT$ such that
\begin{subequations}\label{eq:VMM}
\begin{align}
\intT \la\tddt(\Id+\calQ)\uHT,\vHT\ra + a((\Id+\calQ)\uHT,\vHT) \dt = \intT \la f, \vHT\ra \dt,&\label{eq:VMMcoarse}\\
\intT \la\tddt\Qu, \wht\ra + a(\Qu, \wht) \dt = -\intT \la\duHT,\wht\ra + a(\uHT,\wht) \dt & \label{eq:VMMfine}
\end{align} 
\end{subequations}
for all $\vHT \in \VHT$ and $\wht \in \Wht$. 
Note that~\eqref{eq:VMM} is well-posed, which follows from the fact that~\eqref{eq:VMM} is equivalent to~\eqref{eq:discretizedsol} with slightly adjusted right-hand side $\intT\langle f,\IHT \vht\rangle \dt$ and the existence and uniqueness of the \emph{correction operator} $\calQ\colon\hVht\to\hWht$ in~\eqref{eq:VMMfine} as well as the boundedness of $\IHT$; cf.~Lemma~\ref{lem:boundInterpolation}. Existence and uniqueness of $\calQ$ are a direct consequence of Lemma~\ref{lem:Qwellposed}, which is proved below.
We emphasize that this construction results in an alternative decomposition $\hVht = (\Id+\calQ)\hVHT \oplus \hWht$, where the first space turns out to have favorable approximation properties.
	
With the well-posedness of~\eqref{eq:VMM}, we can state the following theorem that quantifies the error between the solution $\uht \in \hVht$ of~\eqref{eq:discretizedsol} and the solution $\tuHT = (\Id+\calQ)\uHT$ of the ideal method~\eqref{eq:VMM}. 
\begin{theorem}[Error of the ideal method]\label{t:errIdeal}
Assume that $f\in L^2(L^2)\cap H^1(H^{-1})$. Then 
the error between the solutions $\uht$ and $\tuHT$ satisfies
\begin{equation}\label{eq:errIdeal}
\|\tuHT - \uht\|_{\Vtr} \leq C\, (H + \T) \,\|f\|_{L^2(L^2)\cap H^1(H^{-1})}.
\end{equation}
\end{theorem} 

\begin{proof}
Since $\tuHT - \uht \in \hVht$, we get with the inf-sup condition in Lemma~\ref{lem:infsup} the existence of a function $v \in \Vht$ with $\|v\|_{\Vte}=1$ such that
\begin{equation}\label{proof:errIdeal:1}
\begin{aligned}
c_\frakA\,\|\tuHT - \uht\|_{\Vtr}  \leq \frakA(\tuHT - \uht,v)
= - \intT \la f,(\Id-\IHT)v\ra \dt 
\end{aligned}
\end{equation}
where we also use \eqref{eq:VMMcoarse} and \eqref{eq:discretizedsol}. 
To bound the right-hand side of~\eqref{proof:errIdeal:1}, we compute
\begin{align*}
\Big|\intT &\la f,(\Id-\IHT)v\ra \dt\,\Big| = \Big|\intT \la f,(\Id-\IH)v + (\IH-\IHT)v\ra \dt\,\Big|
\\ & \leq \Big|\intT \la f,(\Id-\IH)v \ra \dt\,\Big| + \Big|\sum_{i = 1}^{\NT}\int_{T_{i-1}}^{T_i}\big\la f,\IH v -\T^{-1}\int_{T_{i-1}}^{T_i}\IH v\ds\big\ra \dt\,\Big|
\\ & \leq \intT \|f\|_{\L}\Cint H\,\|\nabla v(t)\|_{\L}\dt + \Big|\sum_{i = 1}^{\NT}\int_{T_{i-1}}^{T_i}\big\la f-\T^{-1}\int_{T_{i-1}}^{T_i} f\ds,\IH v\big\ra \dt\,\Big|
\\ & \leq \intT \|f\|_{\L}\Cint H\,\|\nabla v(t)\|_{\L}\dt + \intT \Cpc \T\,\|\dot f\|_{H^{-1}({\Omega})}\|\nabla\IH v\|_{\L} \dt
\\ & \leq \Cint H\, \|f\|_{L^2(L^2)} + \Cpc\Cint \T\, \|\dot f\|_{L^2(H^{-1})}
\end{align*}
using~\eqref{eq:IH1} and the Poincar\'e inequality with constant $\Cpc$ on multiple sub-intervals of $[0,\frakT]$. Combining this with \eqref{proof:errIdeal:1}, we deduce~\eqref{eq:errIdeal}.
\end{proof}

The combination of Theorem~\ref{t:errIdeal} and~\eqref{eq:erruht} directly provides an estimate for the full error $u - \tuHT$. If $h$ and $\tau$ are chosen such that any fine quantities are resolved, the order $\mathcal{O}(H + \T)$ can be maintained for the full error as well.  

\subsection{Well-posedness of the corrections}
\label{ss:corrprob}	

In order to avoid an explicit characterization of the spaces $\hWht$ and $\Wht$ and to prove well-posedness of the corrections, we reformulate~\eqref{eq:VMMfine} as a constraint problem posed in the full discrete space $\hVht$ and with test functions in $\Vht$. 
For $z \in \hVht$, its correction $\calQ z$ can equivalently be characterized as a~function in $\hVht$ that solves
\begin{subequations}\label{eq:constraintCorProb}
\begin{align}
\frakA(\calQ z,v) + \sum_{i = 1}^{\NT} \int_{T_{i-1}}^{T_i}\la\lambda_i, \IH v\ra \dt &= -\frakA(z,v)\label{eq:constraint1}\\ 
\sum_{j = 1}^{\NT}\la\IH\calQ z(T_j),\mu_j\ra &= 0, \label{eq:constraint2}
\end{align}
\end{subequations}
for all $v \in \Vht,\, \mu_j \in \VH$, $j = 1,\ldots,\NT$, where $(\lambda_1,\ldots,\lambda_{\NT})\in \VH \times \ldots \times \VH$ are the associated Lagrange multipliers.
The following lemma states the well-posedness of~\eqref{eq:constraintCorProb} and therefore also the well-posedness of~\eqref{eq:VMMfine}. 
\begin{lemma}[Existence and uniqueness of the correction]\label{lem:Qwellposed}
Problem~\eqref{eq:constraintCorProb} has a unique solution $(\calQ z,\lambda_1,\ldots,\lambda_{\NT}) \in \hVht \times \VH \times \ldots \times \VH$ and 
\begin{equation}\label{eq:stabPsi}
\|\calQ z\|_{\Vtr} \leq c_\frakA^{-1}
\sqrt{2}\max\{1,\bar\beta\}\,\|z\|_{\Vtr}.
\end{equation}
\end{lemma}

\begin{proof}
Recall the definition of the bilinear form $\frakA$ in~\eqref{eq:defA} and let further 
\begin{align*}
\frakB_i(\lambda_i,w) &:= \int_{T_{i-1}}^{T_i}\la\lambda_i, \IH w\ra \dt,\quad 1 \leq i \leq \NT,\\
\frakC_j(v,\mu_j) &:= \la\IH v(T_j),\mu_j\ra,\quad 1 \leq j \leq \NT.
\end{align*}
By~\cite[Cor.~2.1]{BerCM88}, the well-posedness of~\eqref{eq:constraintCorProb} and the estimate~\eqref{eq:stabPsi} follow from Lemma~\ref{lem:infsup}, the inf-sup conditions~\eqref{eq:infsupB} and~\eqref{eq:infsupC}, as well as an upper bound for the right-hand side of~\eqref{eq:constraint1}, which are proved below.\medskip

\emph{Inf-sup condition for $\sum_i\frakB_i$.}
Let $(\lambda_1,\ldots,\lambda_{\NT}) \in \VH \times \ldots \times \VH$ be arbitrary and nonzero. The explicit piecewise constant choice $w = \sum_{i = 1}^{\NT}\lambda_i \id_{[T_{i-1},T_i]} \in \Vht$ then yields
\begin{align*}
\sup_{w \in \Vht} \frac{\sum_{i = 1}^{\NT}\int_{T_{i-1}}^{T_i}\la\lambda_i, \IH w\ra \dt}{ \|(\lambda_1,\ldots,\lambda_{\NT})\|_{\VH \times \ldots \times \VH}\,\|w\|_{\Vte}} 
& \geq \frac{\sum_{i = 1}^{\NT}\T\,\|\lambda_i\|_{\L}^2}{ \big(\sum_{i = 1}^{\NT}\|\nabla\lambda_i\|^2_{\L}\big)^{1/2}\,\big(\sum_{i = 1}^{\NT}\T\,\|\nabla\lambda_i\|^2_{\L}\big)^{1/2}} 
\\&\geq \Cinv^{-2} H^2 \T^{1/2} > 0
\end{align*} 
using the norm 
$\|(\lambda_1,\ldots,\lambda_{\NT})\|^2_{\VH \times \ldots \times \VH} = \sum_{i = 1}^{\NT} \|\nabla \lambda_i\|^2_{\L}$
and the inverse inequality~\eqref{eq:invIneq}. Therefore, we directly get
\begin{equation}\label{eq:infsupB}
\adjustlimits\inf_{(\lambda_1,\ldots,\lambda_{\NT}) \in \VH \times \ldots \times \VH\,}\sup_{w \in \Vht} \frac{\sum_{i = 1}^{\NT}\frakB_i(\lambda_i,w)}{\|(\lambda_1,\ldots,\lambda_{\NT})\|_{\VH \times \ldots \times \VH}\,\|w\|_{\Vte}} \geq \Cinv^{-2} H^2 \T^{1/2}.
\end{equation}
\medskip

\emph{Inf-sup condition for $\sum_j\frakC_j$.}
To show this inf-sup condition, we choose $(\mu_1,\ldots,\mu_{\NT}) \in \VH \times \ldots \times \VH$ and set $\eta_j(t) = (t/\T - j + 1)\id_{[T_{j-1},T_j]} + (j - t/\T + 1)\id_{(T_{j},T_{j+1}]}$. With the choice $v = \sum_{j=1}^{\NT}\mu_j\eta_j \in \hVht$, we compute 
\begin{align*}
\|v\|^2_{\Vtr} 
&= \intT \big\|\sum_{j = 1}^{\NT}  (\nabla \mu_j) \overline\eta_j \big\|_{\L}^2 + \big\|\sum_{j = 1}^{\NT}\mu_j\dot\eta_j \big\|^2_{H^{-1}({\Omega})} \dt \\
& \leq 2\intT \sum_{j = 1}^{\NT} \big( \|\nabla \mu_j\|_{\L}^2 \overline\eta_j^2 + \| \mu_j\|^2_{\L} \dot\eta_j^2\big) \dt \\
& \leq \sum_{j = 1}^{\NT} \big(2\T\, \|\nabla \mu_j\|_{\L}^2 + \tfrac{4}{\T}\,\| \mu_j\|^2_{\L}\big).
\end{align*}
Therefore, we have that
\begin{equation*}
\|v\|_{\Vtr} = \big\|{\sum_{j=1}^{\NT}\eta_j \mu_j}\big\|_{\Vtr} \leq \Big(2\T\sum_{j = 1}^{\NT} \|\nabla \mu_j\|^2_{\L}\Big)^{1/2} + \Big(\tfrac{4}{\T}\sum_{j = 1}^{\NT} \|\mu_j\|^{{2}}_{\L}\Big)^{1/2}.
\end{equation*}
Using this, Young's inequality, and the inverse inequality~\eqref{eq:invIneq}, we obtain
\begin{align*}
\sup_{v \in \hVht}&\frac{\sum_{j = 1}^{\NT}\la\IH v(T_j),\mu_j\ra}{\|v\|_{\Vtr}\,\|(\mu_1,\ldots,\mu_{\NT})\|_{\VH \times \ldots \times \VH}} \\&\geq \frac{\sum_{j = 1}^{\NT}\|\mu_j\|_{\L}^2}{\|{\sum_{j=1}^{\NT} \mu_j\eta_j}\|_{\Vtr}\,\big(\sum_{j = 1}^{\NT}\|\nabla \mu_j\|_{\L}^2\big)^{1/2}}\\
&\geq \frac{\sum_{j = 1}^{\NT}\|\mu_j\|_{\L}^2}{\big(\T\, \Cinv^2H^{-2} + 2\,\T^{-1} + \Cinv^2 H^{-2}\big)\sum_{j = 1}^{\NT}\|\mu_j\|_{\L}^2}
=: \gamma(H,\T) > 0.
\end{align*}
As above, taking the infimum yields
\begin{equation}\label{eq:infsupC}
\adjustlimits\inf_{(\mu_1,\ldots,\mu_{\NT}) \in \VH \times \ldots \times \VH\,}\sup_{v \in \hVht} \frac{\sum_{j = 1}^{\NT}\frakC_j(v,\mu_j)}{\|v\|_{\Vtr}\,\|(\mu_1,\ldots,\mu_{\NT})\|_{\VH \times \ldots \times \VH}} \geq \gamma(H,\T).
\end{equation}
\medskip

\emph{Bound of the right-hand side.}
The final step consists in proving a stability bound for the right-hand side of~\eqref{eq:constraint1} as a function in the dual space of $\Vte$.  
We choose an arbitrary function $v \in \Vht$ and compute
\begin{align*}
\frakA(z,v)
& \leq \Big(\|\dot z\|_{H^{-1}(\Omega)} + \bar\beta\,\|\nabla \overline z\|_{L^2(\Omega)} \Big)\|v\|_{\Vte}
\leq
\sqrt{2}\max\{1,\bar\beta\}\|z\|_{\Vtr}\|v\|_{\Vte}.
\end{align*}

With the above inf-sup conditions and the bound of the right-hand side, we can now apply~\cite[Cor.~2.1]{BerCM88}, which finalizes the proof.
\end{proof}

\begin{remark}[Additional correction operator]\label{rem:corrC}
The arguments in the proof of Lemma~\ref{lem:Qwellposed} also imply that the correction operator $\calC\colon\Vht \to \Wht$, defined by 
\begin{equation}\label{eq:corrC}
\frakA(v,\calC w) = - \frakA(v, w)
\end{equation}
for all $v \in \hWht$, is well-defined and 
\begin{equation}\label{eq:boundC}
\|\calC w\|_{\Vte} \leq c_{\frakA}^{-1}\sqrt{2}\max\{1,\bar\beta\}\,\|w\|_{\Vte}.
\end{equation}
\end{remark}

With Lemma~\ref{lem:Qwellposed} and Remark~\ref{rem:corrC}, we can show the following inf-sup condition for the spaces $\hWht$ and $\Wht$, which will be of use in the analysis of the localized version of the method in the next section.

\begin{lemma}[Inf-sup condition]\label{lem:infsupW}
It holds
\begin{equation*}
\adjustlimits\inf_{\vht \in \hWht}\sup_{\wht \in \Wht} \frac{\frakA(\vht,\wht)}{\|\vht\|_{\Vtr}\,\|\wht\|_{\Vte}} \geq \tilde c_\frakA 
\end{equation*}
with $\tilde{c}_\frakA := c_{\frakA}^2/(\sqrt{2}\max\{1,\bar\beta\})$. 
\end{lemma} 

\begin{proof}
Let $v \in \hWht$. With~\eqref{eq:corrC} and~\eqref{eq:boundC} and since $\calC$ is surjective, we have that 
\begin{align*}
\sup_{w \in \Wht}\frac{\frakA(v,w)}{\|w\|_{\Vte}} = \sup_{w \in \Vht}\frac{\frakA(v,-\calC w)}{\|\calC w\|_{\Vte}} \geq \sup_{w \in \Vht}\frac{c_{\frakA}\,\frakA(v, w)}{\sqrt{2}\max\{1,\bar\beta\}\|w\|_{\Vte}} \geq \tilde c_\frakA\,\|v\|_{\Vtr},
\end{align*}
using Lemma~\ref{lem:infsup} in the last step.
\end{proof}

\section{A localized multiscale method}\label{s:locmethod}

This section is devoted to introducing a localized variant of the ideal multiscale method presented in Section~\ref{s:msmethod}. We emphasize that the non-localized method is based on auxiliary corrector problems on the entire fine space-time grid, which is not feasible in terms of computational complexity and memory. However, one observes that a corrected function $\calQ z$ decays exponentially fast away from the support of the underlying function $z$, see Section~\ref{ss:decayexample}. Without a great impact on the approximation property, it is therefore possible to restrict the fine-scale computations to local spatial patches around the corresponding node and a limited number of coarse time steps. This is referred to as localization in space and localization in time, respectively. We first discuss the details about how a correction can be computed for one coarse interval at a time, which motivates the temporal localization. Then we introduce the local patches to which we restrict the computation of the corrections for all coarse basis functions. Finally, we combine the two ideas and end this section by stating a space- and time-localized multiscale method and presenting corresponding error bounds.

We emphasize that in view of~\eqref{eq:VMMfine} and due to linearity, only so-called \emph{basis correctors} need to be computed, i.e., $\calQ\Lambda_x^i$, where $\Lambda^i_x=\varphi_x\zeta_i \in \hVHT,\,x\in \calNH,\,i \in\{1,\ldots, \NT\}$ is a coarse (nodal) space-time basis function with $\supp(\Lambda^i_x) = \Nb(x) \times [T_{i-1},T_{i+1}]$. In the following, we often abbreviate $\Lambda = \Lambda_x^i$ when referring to one particular basis function.

\subsection{Localization in time} 

Let $\Lambda \in \hVHT$ be a space-time basis function as above. In this subsection, we construct $\calQ\Lambda$ by a sequential approach. 
We divide the integral in~\eqref{eq:constraintCorProb} into local integrals over $[T_{j-1},T_{j}],\, j=i,\ldots,\NT$ and define for given $j$ the local version of $\hVht$ by 
\begin{equation*}
\hVhtj := \{v\id_{\Omega \times [T_{j-1},T_j]}\,\colon\, v \in \hVht\}.
\end{equation*} 
Further, we denote with $\xi_j \in \hVhtj,\,\xi_j(T_{j-1}) = 0$ the solution to the auxiliary problem
\begin{subequations}\label{eq:xi}
\begin{align}
\int_{T_{j-1}}^{T_j} \la\dot{\xi}_{j}, v\ra + a(\xi_{j}, v) + \la\lambda_j, \IH v\ra \dt&= -\int_{T_{j-1}}^{T_j} \la\dot{\Lambda}, v\ra + a(\Lambda,v) - \la\tfrac{1}{T}\xi_{j-1}(T_{j-1}), v\ra \nonumber\\ 
&\qquad \qquad \qquad + a(\tfrac{T_j-t}{\T} \xi_{j-1}(T_{j-1}),v) \dt, \label{eq:xi_1}\\
\la\IH\xi_{j}(T_j), \mu\ra &= 0,
\end{align}
\end{subequations}
for all $v\in \Vht$, $\mu \in \VH$, where $\lambda_j \in \VH$ is the associated Lagrange multiplier. For $j=i$, we explicitly set $\xi_{i-1}(T_{i-1}) = 0$ such that the third and the fourth term on the right-hand side of~\eqref{eq:xi_1} vanish. Note that the functions $\{\xi_j\}_{j=i}^{\NT}$ are constructed such that 
\begin{equation*}
\calQ\Lambda= \sum_{j=i}^{\NT} \big(\xi_j + \tfrac{T_j-t}{\T}\xi_{j-1}(T_{j-1})\big)\id_{[T_{j-1},T_j]}.
\end{equation*}
We emphasize that $\Lambda$ only has support on $[T_{i-1},T_{i+1}]$. That is, \mbox{$\calQ\Lambda\,\id_{[0,T_{i-1}]} \equiv 0$} and for $j\geq2+i$ the first two terms in~\eqref{eq:xi_1} (and also in~\eqref{eq:constraint1}) disappear. Consequently, $\calQ\Lambda$ will begin to decay due to the parabolic nature of the problem.

Due to the decay property of $\calQ\Lambda$, there will be an $\ell \in \mathbb{N}$ such that for $j\geq\ell+i$, the sequential functions $\xi_j$ will be of negligible size compared to the error of the ideal method. Hence, it suffices to restrict the computations to $\{\xi_j\}_{j=i}^{\ell+i-1}$. That is, we choose the temporally localized corrector function as $\calQ\Lambda\,\id_{[0,T_{\ell+i-1}]}$. From here on, we will refer to $\ell$ as the \emph{temporal localization parameter}. We remark that simply restricting the computations will make the function discontinuous in time, and thus it will no longer be a function in $\hVht$. However, this can easily be circumvented by setting
\begin{equation*}
\calQ_{\infty,\ell}\Lambda := \calQ\Lambda\,\id_{[0,T_{\ell+i-1}]} + \tfrac{T_{\ell+i}-t}{\T} \calQ\Lambda(T_{\ell+i-1})\,\id_{(T_{\ell+i-1},T_{\ell+i}]},
\end{equation*} 
i.e., we extend the solution by one time step, where it linearly goes down to zero. 

\subsection{Localization in space}

For the spatial localization of \eqref{eq:constraintCorProb}, we first introduce the element-based patches to which the support of $\calQ\Lambda$ is to be restricted. Given $\Nb(K)$ as defined in \eqref{eq:NT} for an element $K\in \calTH$, we define the patch $\Nb^k(K)$ of size $k$ as
\begin{align*}
\Nb^1(K) := \Nb(K), \qquad
\Nb^k(K) := \Nb(\Nb^{k-1}(K)), \ \text{for $k\geq 2$}.
\end{align*}
An example of how the patches spread across the grid with increasing $k$ is illustrated in Figure \ref{fig:patches}.
With these coarse element patches defined, we can as well define localized fine-scale spaces for an element $K \in \calTH$ by
\begin{align*}
\hVhtkK &:= \{w\in \hVht: \supp(w) \subseteq \Nb^k(K)\}, \\
\VhtkK &:= \{w\in \Vht: \supp(w) \subseteq \Nb^k(K)\},
\end{align*}
and analogously the localized remainder spaces $\hWhtkK$ and $\WhtkK$, as well as $\VHkK$ and $\VhkK$.
Let now $i \in \{1,\ldots,\NT\}$ and $\DKi = K \times [T_{i-1},T_{i}]$. 
For such a space-time element, we introduce the element restricted correction operator $\QKi v\in \hWht$ that solves 
\begin{align}\label{eq:defQKi}
\intT \la \tddt\QKi v, w \ra + a(\QKi v, w) \dt = -\int_{T_{i-1}}^{T_i} \la \dot{v},w \ra_K + a(v,w)_K \dt, &\quad w \in \Wht,
\end{align}
where the subscript $K$ on the right-hand side indicates that corresponding integrals are taken over the element $K\in \calTH$ instead of the entire domain ${\Omega}$. As in Section~\ref{ss:corrprob}, these problems can also be defined as constraint problems and are well-posed. Note that we retain the global correction operator by summing all local contributions, i.e.,
\begin{align*}
\calQ v = \sum_{K \in \calTH,\,1 \leq i \leq \NT} \QKi v.
\end{align*}
For $k\in \mathbb{N}$, we may restrict the correction to an element patch by defining $\QkKi v\in \hWhtkK$ as the solution to
\begin{align*}
\intT \la \tddt\QkKi v, w \ra + a(\QkKi v, w) \dt = -\int_{T_{i-1}}^{T_i} \la \dot{v},w\ra_K + a(v,w)_K \dt, &\quad w \in \WhtkK.
\end{align*}
We emphasize that solvability of this localized problem follows in analogy with the global case. Once again, we sum over all elements to obtain a corresponding global version by
\begin{align}\label{eq:locCor}
\calQ_{k,\infty} v := \sum_{K \in \calTH,\,1 \leq i \leq \NT} \QkKi v.
\end{align}
Having defined the spatially localized correction operator $\calQ_{k,\infty}$, we can replace $\calQ\Lambda$ by $\calQ_{k,\infty}\Lambda$ in the construction of the multiscale method. In the following, we refer to $k$ as the \emph{spatial localization parameter}.

\begin{figure}
\centering
\begin{tikzpicture} [scale=0.5]

\draw[fill=black!30!white] (2,2) -- (5,2) -- (7,4) -- (7,7) -- (5,7) -- (2,4) -- (2,2);

\draw[fill=black!60!white] (3,3) -- (5,3) -- (6,4) -- (6,6) -- (5,6) -- (3,4) -- (3,3);

\fill[fill=black!100!white] (4,4) -- (5,4) -- (5,5) -- (4,4);

\draw [line width=0.2mm, draw=black, fill=black!20!white] (1,1) grid  (8,8);

\foreach \i in {1,...,7}
{
	\draw[line width=0.2mm, draw=black, fill=black!20!white] (1,8-\i) -- (1+\i,8);
	\draw[line width=0.2mm, draw=black, fill=black!20!white] (8-\i,1) -- (8,1+\i);
}

\node[scale=1.2] at (4.5,8.7) {$\calTH$};

\fill[fill=black!100!white] (4+\xf,4+\yf) -- (5+\xf,4+\yf) -- (5+\xf,5+\yf) -- (4+\xf,4+\yf);

\draw[fill=black!60!white] (3+\xs,3+\ys) -- (5+\xs,3+\ys) -- (6+\xs,4+\ys) -- (6+\xs,6+\ys) -- (5+\xs,6+\ys) -- (3+\xs,4+\ys) -- (3+\xs,3+\ys);

\foreach \i in {0,1}
{
	\draw[line width=0.2mm, draw=black, fill=black!20!white] (3+\xs+\i,3+\ys) -- (6+\xs,6+\ys-\i);
	\draw[line width=0.2mm, draw=black, fill=black!20!white] (4+\xs+\i,3+\ys) -- (4+\xs+\i,5+\ys+\i);
	\draw[line width=0.2mm, draw=black, fill=black!20!white] (3+\xs+\i,4+\ys+\i) -- (6+\xs,4+\ys+\i);
}

\draw[fill=black!30!white] (2+\xt,2+\yt) -- (5+\xt,2+\yt) -- (7+\xt,4+\yt) -- (7+\xt,7+\yt) -- (5+\xt,7+\yt) -- (2+\xt,4+\yt) -- (2+\xt,2+\yt);

\foreach \i in {0,1,2}
{
	\draw[line width=0.2mm, draw=black, fill=black!20!white] (2+\xt+\i,2+\yt) -- (7+\xt,7+\yt-\i);
	\draw[line width=0.2mm, draw=black, fill=black!20!white] (3+\xt+\i,2+\yt) -- (3+\xt+\i,5+\yt+\i);
	\draw[line width=0.2mm, draw=black, fill=black!20!white] (2+\xt+\i,4+\yt+\i) -- (7+\xt,4+\yt+\i);
}
\draw[line width=0.2mm, draw=black, fill=black!20!white] (2+\xt,3+\yt) -- (6+\xt,3+\yt);
\draw[line width=0.2mm, draw=black, fill=black!20!white] (2+\xt,3+\yt) -- (6+\xt,7+\yt);
\draw[line width=0.2mm, draw=black, fill=black!20!white] (6+\xt,3+\yt) -- (6+\xt,7+\yt);

\node[scale=1.2] at (5.0+\xf,5.7+\yf) {$K$};
\node[scale=1.2] at (5.0+\xs,6.7+\ys) {$\Nb^1(K)$};
\node[scale=1.2] at (5.0+\xt,7.7+\yt) {$\Nb^2(K)$};

\end{tikzpicture} 
\caption{\small Illustration of patches around an element $K$.}
\label{fig:patches}
\end{figure}
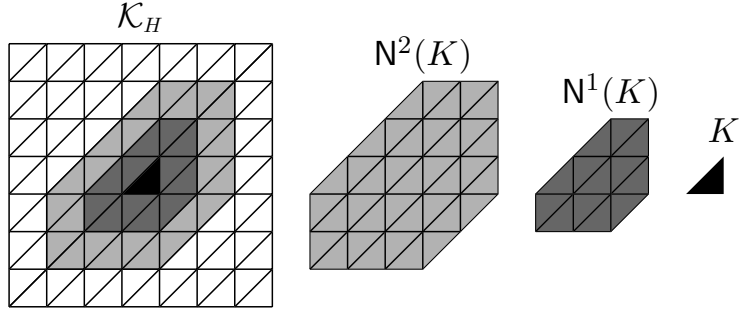
	
\subsection{Localization in space and time}\label{ss:locst}
We can now combine the spatial and temporal localization procedures from the previous subsections to create a localized variant of the space-time multiscale method introduced in Section~\ref{s:msmethod}.

Let $k\in \mathbb{N}$ and $\ell\in \mathbb{N}$ be given localization parameters in space and time, respectively. Further, we define for $j = 1,\ldots \NT$
and any $K\in \calTH$ the restricted version of $\hVhtkK$ by
\begin{equation*}
\hVhtkKj := \{v\id_{\Omega \times [T_{j-1},T_j]}\,\colon\, v \in \hVhtkK\}.
\end{equation*} 
Now let $\DKi = K \times [T_{i-1},T_{i}]$ be a fixed space-time element and $\Lambda \in \hVHT$ a nodal basis function whose support overlaps with $\DKi$.  
Further, let $\xiK \in \hVhtkKj,\,\xiK(T_{j-1}) = 0$ be the solution to
\begin{subequations}\label{eq:xiK}
\begin{align}
\int_{T_{j-1}}^{T_j} \la\tddt\xiK, v\ra + a(&\xiK, v) + \la\lambda_{j,k}^{\scalebox{0.6}{$K$},i}, \IH v\ra \dt \nonumber\\&= -\int_{T_{i-1}}^{T_i} \la\dot{\Lambda}, v\ra_K + a(\Lambda,v)_K \dt \\ 
&\qquad \qquad + \int_{T_{j-1}}^{T_j} \la\tfrac{1}{\T}\xiKp(T_{j-1}), v\ra  - a(\tfrac{T_j-t}{\T} \xiKp(T_{j-1}),v) \dt, \nonumber\\
\la\IH\xiK(T_j), \mu\ra &= 0
\end{align}
\end{subequations}
for all $v\in \VhtkK$ and $\mu \in \VH$, where $\lambda_{j,k}^{\scalebox{0.6}{$K$},i} \in \VHkK$ is the associated Lagrange multiplier. Moreover, we have the initial condition $\xi^{\scalebox{.6}{$K$},i}_{i-1,k}(T_{i-1}) = 0$. The localized space-time basis corrector $\Qkl\Lambda$ is then defined as 
\begin{equation}\label{eq:loclocCor}
\begin{aligned}
\Qkl\Lambda &= \sum_{\stackrel{K\in \calTH,\,1 \leq i \leq \NT:}{\DKi\,\cap\, \supp(\Lambda) \neq \emptyset}} \QklKi \left(\Lambda\id_{\DKi}\right)\\
&:= \sum_{\stackrel{K\in \calTH,\,1 \leq i \leq \NT:}{\DKi\,\cap\, \supp(\Lambda) \neq \emptyset}}\Big( \sum_{j=i}^{{i+\ell-1}} \big( \xiK + \tfrac{T_{j}-t}{\T}\xiKp(T_{j-1}) \big)\id_{(T_{j-1},T_j]} \\&\qquad\qquad\qquad\qquad\qquad\qquad\qquad +\tfrac{T_{{i+\ell}}-t}{\T} \xiKl(T_{i+\ell-1})\,\id_{(T_{i+\ell-1},T_{{i+\ell}}]}\Big).
\end{aligned}
\end{equation}
We emphasize that this construction also allows to define $\QklKi$, and thus $\Qkl$, for any function $\vHT \in\hVHT$ by replacing $\Lambda$ by $\vHT$ in~\eqref{eq:xiK} and~\eqref{eq:loclocCor}.

\subsection{Localized space-time multiscale method}
With the localized correction operator $\Qkl$ defined in the previous subsection, the proposed \emph{localized multiscale method} reads: find $\tuHTkl = \uHTkl + \Qkl \uHTkl \in (\Id+\Qkl)\hVHT$ such that
\begin{equation}
\intT \la\tddt(\Id+\Qkl)\uHTkl,\vHT\ra + a((\Id+\Qkl)\uHTkl,\vHT) \dt = \intT \la f, \vHT\ra \dt \label{eq:loccoarseproblem}
\end{equation}
for all $\vHT \in \VHT$.\medskip

In the following subsection, we investigate the localization procedure by means of an a posteriori error estimate and then investigate the error of the localized multiscale method.

\subsection{Well-posedness and error of the localized method}

Let $\DKi = K \times (T_{i-1},T_{i})$, $K \in \calTH,\,i = 1,\ldots,\NT$ be local space-time elements. Further, we define an error estimator for the spatial decay,
\begin{equation}\label{eq:errEstX}
\delta^{\DKi}_{k,\ell} := \sup_{\vHT \in \VHT}\frac{\|\QklKi(\vHT\id_{\DKi})\id_{\Nb^{k}(K)\setminus\Nb^{k-3}(K)}\|_{\Vtr}}{\|\vHT\id_{\DKi}\|_{\Vtr}}, 
\end{equation}
an error estimator for the temporal decay,
\begin{equation}\label{eq:errEstT}
\vartheta^{\DKi}_{k,\ell} := \big(\T^{-1/2}H + \T^{1/2}\big)\,\sup_{\vHT \in \VHT}\frac{\|\nabla\QklKi(\vHT\id_{\DKi})(T_{{i+\ell-1}})\|_{\L}}{\|\vHT\id_{\DKi}\|_{\Vtr}},
\end{equation}
as well as 
\begin{equation*}
\delta_{k,\ell} := \adjustlimits\sup_{K \in \calTH}\sup_{i = 1,\ldots,\NT} \delta^{\DKi}_{k,\ell}, \qquad \vartheta_{k,\ell} := \adjustlimits\sup_{K \in \calTH}\sup_{i = 1,\ldots,\NT} \vartheta^{\DKi}_{k,\ell}.
\end{equation*}
Since the corrections $\QklKi$ numerically show a rapid decay in space and time away from the element $\DKi$, we expect the spatial and temporal error indicator to decay as well when $k$ and $\ell$ are increased. This is due to the fact that the ring $\Nb^{k}(K)\setminus\Nb^{k-3}(K)$ is further away from the element $K$ if $k$ is increased and $T_{i+\ell-1}$ is further away from $[T_{i-1},T_i]$ if $\ell$ is increased. 

We emphasize that $\delta_{k,\ell}$ and $\vartheta_{k,\ell}$ can be computed by solving small (coarse-scale) eigenvalue problems. In particular, we can adjust the localization parameters $k$ and $\ell$ in order to obtain numbers $\delta_{k,\ell}$ and $\vartheta_{k,\ell}$ that are smaller than a certain threshold. We employ these estimators in the a posteriori bounds derived below. 

\begin{lemma}[A posteriori localization error]\label{lem:aposteriori}
Let $\calQ$ be the correction operator defined in~\eqref{eq:VMMfine} and $\Qkl$ its localized version defined by~\eqref{eq:loclocCor} with $k \geq 3$ and $\ell \geq 1$. Then 
\begin{equation*}
\begin{aligned}
\|(\calQ-&\Qkl)\vHT\|_{\Vtr}
\leq \tilde c_\frakA^{-1}\bar\beta\,\Cint(C_\eta H^{-1} + 1) \,C \big(k^{(d-1)/2}\ell^{1/2}\,\delta_{k,\ell} + k^{d/2}\,\vartheta_{k,\ell})\,\|\vHT\|_{\Vtr}.
\end{aligned}
\end{equation*}
\end{lemma}

\begin{proof}
Let $K \in \calTH$ and $1 \leq i \leq \NT$. Further, let $\QKi$ be the correction operator defined in~\eqref{eq:defQKi} and $\QklKi$ its space- and time-localized variant as given in~\eqref{eq:loclocCor}. Due to the definition of $\QKi$ and $\QklKi$, we have that 
\begin{equation*}
(\QKi-\QklKi)(\vHT\id_{\DKi}) \in \hWht, 
\end{equation*}
where $\DKi = K \times (T_{{i-1}},T_{{i}})$ as above. Now, let $\eta \in C^0(\Omega)$ be a cutoff function with 
\begin{equation*}
0 \leq \eta \leq 1,\quad \eta\id_{\Omega\setminus\Nb^{k-1}(K)} \equiv 1,\quad \eta\id_{\Nb^{k-2}(K)}\equiv 0,\quad \|\nabla \eta\|_{L^\infty(\Omega)} \leq C_\eta H^{-1}.
\end{equation*}
For the moment, we abbreviate $\rho := \QKi(\vHT\id_{\DKi})$, $\rho_{k,\ell} := \QklKi(\vHT\id_{\DKi})$, and $v := \vHT\id_{\DKi}$. Further, we write $R^k(K) := \Nb^k(K)\setminus\Nb^{k-3}(K)$. 
Let $I_h\colon C^0(\Omega)\to \Vh$ be the (spatial) nodal interpolation operator with respect to $\calTh$. Note that $I_h \circ I_h = I_h$ and, by standard interpolation and inverse estimates, there exists a constant $C_I$ such that
\begin{equation*}
	\|\nabla (I_h q)\|_{L^2(K)} \leq C_I\,\|\nabla q\|_{L^2(K)}
\end{equation*}
for any $K \in \calTH$ and any piecewise quadratic polynomial $q$ (with respect to the mesh~$\calTh$).	
Let now $w \in \Wht$. Observe that $I_h w = w$, $\IHT w = 0$, and thus
\begin{equation*}
	w = (\Id-\IHT)w = (\Id-\IHT)I_h w = (\Id-\IHT)I_h(\eta w) + (\Id-\IHT)I_h((1-\eta) w).
\end{equation*}
Using this equality, we compute 
\begin{align}
\frakA(\rho - \rho_{k,\ell},w) 
& = 
\frakA(\rho-\rho_{k,\ell},(\Id-\IHT)I_h(\eta w) + (\Id-\IHT)I_h((1-\eta)w))\nonumber\\
& = \underbrace{-\frakA(v,(\Id-\IHT)I_h(\eta w))}_{=\,0} + \frakA(v,(\Id-\IHT)I_h((1-\eta) w))\nonumber\\
&\qquad - \frakA(v,(\Id-\IHT)I_h((1-\eta) w) ) - \frakA(\rho_{k,\ell},(\Id-\IHT)I_h(\eta w)) \nonumber\\ &\qquad- \frakA(\rho_{k,\ell}\id_{[T_{{i+\ell-1}},T_{{i+\ell}}]},(\Id-\IHT)I_h({(1-\eta)} w) ) \label{eq:proof_loc_1}\\
& = - \frakA(\rho_{k,\ell},(\Id-\IHT)I_h(\eta w)) 
- \frakA(\rho_{k,\ell}\id_{[T_{{i+\ell-1}},T_{{i+\ell}}]},(\Id-\IHT)I_h((1-\eta) w) )\nonumber\\
& \leq \bar\beta\Big( \big\|\rho_{k,\ell}\id_{R^k(K)}\id_{[T_{{i-1}},T_{{i+\ell}}]}\big\|_{\Vtr} 
\,\big\|(\Id-\IHT)I_h(\eta w)\id_{R^k(K)}\id_{[T_{{i-1}},T_{{i+\ell}}]}\big\|_{\Vte} \nonumber\\
&\qquad + \big\|\rho_{k,\ell}
\id_{[T_{{i+\ell-1}},T_{{i+\ell}}]}\big\|_{\Vtr}\, \big\|(\Id-\IHT)I_h((1-\eta)w) \id_{\Nb^k(K)}
\id_{[T_{{i+\ell-1}},T_{{i+\ell}}]}\big\|_{\Vte} \Big).\nonumber
\end{align}
Note that the supports in space come from the fact that $\IHT$ extends the spatial support by at most one layer of elements. 

Next, we bound the terms on the right-hand side of~\eqref{eq:proof_loc_1}. 
Using the definition of $\rho_{k,\ell}$ and~\eqref{eq:IH1}, we get
\begin{equation}\label{eq:proof_loc_2}
\begin{aligned}
\|\rho_{k,\ell}\id_{[T_{i+\ell-1},T_{i+\ell}]}\|_{\Vtr} &= \|\tfrac{T_{i+\ell}-t}{\T}\rho_{k,\ell}(T_{i+\ell-1}) \,\id_{[T_{i+\ell-1},T_{i+\ell}]}\|_{\Vtr} \\
& \leq \bigg(\frac{1}{\T}\Big(\int_{T_{i+\ell-1}}^{T_{i+\ell}} \|\rho_{k,\ell}(T_{i+\ell-1})\|^2_{\L} \dt\Big)^{1/2} \\&\qquad+ \Big(\int_{T_{i+\ell-1}}^{T_{i+\ell}} \|\nabla\overline{\tfrac{T_{i+\ell}-t}{\T}\rho_{k,\ell}(T_{i+\ell-1})}\|^2_{\L} \dt\Big)^{1/2}\bigg)\\
& \leq \big(\Cint\,\T^{-1/2}H + \frac{1}{\sqrt{2}}\, \T^{1/2}\big)\|\nabla\rho_{k,\ell}(T_{i+\ell-1})\|_{\L}\\
& \leq C\big(\T^{-1/2}H + \T^{1/2} \big)\|\nabla\rho_{k,\ell}(T_{i+\ell-1})\|_{\L}.
\end{aligned}
\end{equation}
The last ingredient for the final estimate is a bound of $\|(\Id-\IHT)(\eta w)\id_D\|_{\Vte}$ for some space-time subdomain $D$. Since $\IHT\circ\IHT = \IHT$, we have that $\|(\Id-\IHT)\|_{\mathcal{L}(\Vte)} = \|\IHT\|_{\mathcal{L}(\Vte)}$; see, e.g., \cite{XuZ03}. With the product rule, the bounds on $\eta$, the Friedrichs inequality, and~\eqref{eq:stabIHT}, we therefore obtain
\begin{equation}\label{eq:proof_loc_4}
\begin{aligned}
\|(\Id-\IHT)I_h(\eta w)\id_D\|_{\Vte} &\leq \|(\Id-\IHT)\|_{\mathcal{L}(\Vte)}\,C_I \|\eta w\id_D\|_{\Vte}\\
&\leq \|\IHT\|_{\mathcal{L}(\Vte)} (C_\eta H^{-1}+1)C_I\|w\id_D\|_{\Vte}\\
&\leq \Cint(C_\eta H^{-1}+1)C_I\|w\id_D\|_{\Vte}.
\end{aligned}
\end{equation}
Going back to~\eqref{eq:proof_loc_1} and using \eqref{eq:proof_loc_2}, \eqref{eq:proof_loc_4}, as well as the error estimators defined in~\eqref{eq:errEstX} and~\eqref{eq:errEstT}, we obtain
\begin{align}
\frakA(\QKi(\vHT&\id_{\DKi})-\QklKi(\vHT\id_{\DKi}),w) \nonumber\\&\leq \bar\beta\,\Cint(C_\eta H^{-1} + 1)C_I\Big(\delta^{\DKi}_{k,\ell}\,\big\| w\id_{R^k(K)}\id_{[T_{{i-1}},T_{i+\ell}]} \big\|_{\Vte}
\label{eq:proof_loc_5}\\&\hspace{2cm}+ C\,\vartheta^{\DKi}_{k,\ell}\, \big\|w \id_{\Nb^k(K)}
\id_{[T_{i+\ell-1},T_{i+\ell}]}\big\|_{\Vte} \Big) \,\big\|\vHT\id_{\DKi}\big\|_{\Vtr}.\nonumber
\end{align}
Finally, we sum~\eqref{eq:proof_loc_5} over all $K \in \calTH$ and $i = 1,\ldots \NT$. Therefore, we note that with Lemma~\ref{lem:infsupW} there exists a function $w \in \Wht,\, \|w\|_{\Vte} = 1$ such that
\begin{align*}
\tilde c_\frakA\,\|(\calQ-\Qkl)\vHT\|_{\Vtr} &\leq \frakA((\calQ-\Qkl)\vHT,w)
\\&
= \sum_{K \in \calTH}\sum_{i = 1}^{\NT} \frakA((\QKi-\QklKi)(\vHT\id_{\DKi}),w)\\
&\leq \bar\beta\,\Cint(C_\eta H^{-1} + 1)C_I\sum_{K \in \calTH}\sum_{i = 1}^{\NT}\Big(\delta^{\DKi}_{k,\ell}\,\big\| w\id_{R^k(K)}\id_{[T_{{i-1}},T_{i+\ell}]} \big\|_{\Vte}
\\&\hspace{1.5cm}+ C\,\vartheta^{\DKi}_{k,\ell}\, \big\|w \id_{\Nb^k(K)}
\id_{[T_{i+\ell-1},T_{i+\ell}]}\big\|_{\Vte} \Big) \,\big\|\vHT\id_{\DKi}\big\|_{\Vtr}\\
& \leq \bar\beta\,\Cint(C_\eta H^{-1} + 1) \,C \big(k^{(d-1)/2}\ell^{1/2}\,\delta_{k,\ell} + k^{d/2}\,\vartheta_{k,\ell}\big)\,\|\vHT\|_{\Vtr},
\end{align*}
where we use a discrete Cauchy--Schwarz inequality and $\|w\|_{\Vte} = 1$ in the last step. 
We remark that $w$ in the first and second term on the left-hand side of~\eqref{eq:proof_loc_5} is supported on $\mathcal{O} (k^{d-1}\ell)$ and $\mathcal{O}(k^d)$ elements, respectively, which leads to the stated pre-factors due to the global overlap. 
\end{proof}

\begin{remark}[Localization error]
For fixed choices of the localization parameters $k$ and $\ell$, the error indicators $\delta_{k,\ell}$ and $\vartheta_{k,\ell}$ can be explicitly computed without much effort. For large enough localization parameters, we expect these values to be reasonably small such that the error estimate in Lemma~\ref{lem:aposteriori} is of order $\mathcal{O}(H + \T)$. This is expected by our numerical experiments, which indicate an exponential decay in both $k$ and $\ell$. 
Finally, we emphasize that decay in space is theoretically and practically observed in the elliptic setting~\cite{MalP14} and, additionally, parabolic equations naturally decay exponentially in time~\cite[Ch.~7]{KnaA03}. 
\end{remark}

Provided that the localization error estimated in Lemma~\ref{lem:aposteriori} is sufficiently small (which can be verified using the estimators $\delta_{k,\ell}$ and $\vartheta_{k,\ell}$), we can now also provide an a posteriori justification for the well-posedness of the localized multiscale method given in~\eqref{eq:loccoarseproblem}. 

\begin{lemma}[A posteriori inf-sup condition]\label{lem:apinfsup}
Let $\T \leq C H$ and assume that $k$ and $\ell$ are large enough such that  
\begin{equation}\label{eq:locbound}
\|(\calQ-\Qkl)\vHT\|_{\Vtr} \leq \frac{c_\frakA}{2 \bar\beta\, \hCint\,\Cint}H\, \|\vHT\|_{\Vtr}. 
\end{equation}
Then, the following inf-sup condition holds,
\begin{equation}\label{eq:infsupLocSpace}
\adjustlimits\inf_{\vHT \in \hVHT}\sup_{\wHT \in \VHT} \frac{\frakA((\Id+\Qkl)\vHT,\wHT)}{\|(\Id + \Qkl)\vHT\|_{\Vtr}\,\|\wHT\|_{\Vte}} \geq \hat c_\frakA
\end{equation}
with $\hat c_\frakA := c_\frakA/(2\, \hCint\Cint)$. 
In particular, problem~\eqref{eq:loccoarseproblem} is well-posed. 
\end{lemma}

\begin{proof}	
Let $\vHT \in \hVHT$. 
Note that with~\eqref{eq:stabhIHT} and the definition of $\Qkl$, we have that 
\begin{equation}\label{eq:proof_is_1}
\|\vHT\|_{\Vtr} = \|\hIHT(\Id + \Qkl)\vHT\|_{\Vtr} \leq \hCint H^{-1}\,\|(\Id + \Qkl)\vHT\|_{\Vtr}.
\end{equation}
Further, by Lemma~\ref{lem:infsup}, there exists a function $w \in \Vht$ such that
\begin{equation*}
\frakA((\Id+\Qkl)\vHT,w) \geq c_\frakA\,\|(\Id+\Qkl)\vHT\|_{\Vtr}\,\|w\|_{\Vte}.
\end{equation*} 
Using this, \eqref{eq:locbound}, and~\eqref{eq:proof_is_1}, we compute
\begin{align*}
\sup_{\wHT \in \VHT}\frac{\frakA((\Id+\Qkl)\vHT,\wHT)}{\|(\Id+\Qkl)\vHT\|_{\Vtr}\,\|\wHT\|_{\Vte}} &\geq \frac{\frakA((\Id+\Qkl)\vHT, \IHT w)}{\|(\Id+\Qkl)\vHT\|_{\Vtr}\,\|\IHT w\|_{\Vte}} \\
&\geq \frac{\frakA((\Id+\Qkl)\vHT, w)}{\Cint\,\|(\Id+\Qkl)\vHT\|_{\Vtr}\,\| w\|_{\Vte}} \\&\qquad\qquad-  \frac{|\frakA((\calQ-\Qkl)\vHT,(\Id - \IHT) w)|}{\Cint\,\|(\Id+\Qkl)\vHT\|_{\Vtr}\,\| w\|_{\Vte}}\\
&\geq \frac{c_\frakA}{ \Cint} - \frac{c_\frakA}{2\, \Cint}	= \frac{c_\frakA}{2\, \Cint}. \qedhere 
\end{align*}
\end{proof}
With Lemma~\ref{lem:aposteriori} and Lemma~\ref{lem:apinfsup}, we can now quantify the error of the proposed localized multiscale method.

\begin{theorem}[Error of the practical method] Suppose that the assumptions of Lemma~\ref{lem:apinfsup} hold. Further, let $\uht \in \hVht$ be the solution to~\eqref{eq:discretizedsol}, $\uHT$ the coarse part of the solution to~\eqref{eq:VMM}, and $\tuHTkl = (\Id + \Qkl)\uHTkl$ the solution to~\eqref{eq:loccoarseproblem}. Then
\begin{equation}\label{eq:finalerr}
\begin{aligned}
\| \uht - \tuHTkl \|_{\Vtr} \lesssim (H + \T)\,\|f&\|_{L^2(L^2) \cap H^1(H^{-1})} \\&
+ H^{-1}\big(k^{(d-1)/2}\ell^{1/2}\delta_{k,\ell} + k^{d/2} \vartheta_{k,\ell}\big)\,\|\uHT\|_{\Vtr}.
\end{aligned}
\end{equation}
That is, if $k$ and $\ell$ are chosen large enough and with the decay of the error estimators, we retain a convergence rate of order $\mathcal{O}(H+\T)$.
\end{theorem}

\begin{proof}
Let $\Rht\colon \hVht \to (\Id+\calQ_{k,\ell})\hVHT$ be the Ritz projection defined for $\vht \in \hVht$ by
\begin{equation}\label{eq:proof_error_1}
\frakA(\Rht\vht,\wHT) = \frakA(\vht,\wHT)
\end{equation}
for all $\wHT \in \VHT$.
Since $\Rht \circ \Rht = \Rht$, we have $\|\Id-\Rht\|_{\mathcal{L}(\hVht)} = \|\Rht\|_{\mathcal{L}(\hVht)}$; see, e.g.~\cite{XuZ03}.
Using this, we compute
\begin{equation}\label{eq:proof_error_2}
\begin{aligned}
\| \uht - (\Id+\Qkl)\uHTkl \|_{\Vtr} & = \| (\Id - \Rht) \uht \|_{\Vtr} \\
& = \| (\Id - \Rht) (\uht - (\Id+\Qkl)\uHT) \|_{\Vtr}\\
& \leq \|\Rht\|_{\mathcal{L}(\hVht)} \, \|\uht - (\Id+\Qkl)\uHT \|_{\Vtr},
\end{aligned}
\end{equation}
where $\uHT \in \hVHT$ is the coarse-scale part of the solution to the ideal method~\eqref{eq:VMM}. 
With the inf-sup condition~\eqref{eq:infsupLocSpace} and \eqref{eq:proof_error_1}, we obtain
\begin{equation}\label{eq:proof_error_3}
\begin{aligned}
\|\Rht\|_{\mathcal{L}(\hVht)} &= \sup_{\vht \in \hVht} \frac{\|\Rht \vht\|_{\Vtr}}{\|\vht\|_{\Vtr}} \\
&\leq \hat c_\frakA^{-1} \adjustlimits\sup_{\vht \in \hVht}\sup_{\wHT \in \VHT} \frac{\frakA(\Rht \vht, \wHT)}{\|\vht\|_{\Vtr}\,\|\wHT\|_{\Vte}} \\
&= \hat c_\frakA^{-1} \adjustlimits\sup_{\vht \in \hVht}\sup_{\wHT \in \VHT} \frac{\frakA( \vht, \wHT)}{\|\vht\|_{\Vtr}\,\|\wHT\|_{\Vte}} \\
&\leq \hat c_\frakA^{-1} \sqrt{2}\max\{1,\bar\beta\}.
\end{aligned}
\end{equation}
Using Theorem~\ref{t:errIdeal} and Lemma~\ref{lem:aposteriori}, we further get
\begin{equation}\label{eq:proof_error_4}
\begin{aligned}
\|\uht - (\Id+\Qkl)\uHT \|_{\Vtr} &\leq \|\uht - (\Id+\calQ)\uHT \|_{\Vtr} + \|(\calQ - \Qkl)\uHT \|_{\Vtr}\\
&\lesssim (H + \T)\,\|f\|_{L^2(L^2) \cap H^1(H^{-1})} \\&\qquad+ H^{-1}\big(k^{(d-1)/2}\ell^{1/2}\delta_{k,\ell} + k^{d/2} \vartheta_{k,\ell}\big)\,\|\uHT\|_{\Vtr}. 
\end{aligned}
\end{equation}
The combination of~\eqref{eq:proof_error_2}--\eqref{eq:proof_error_4} completes the proof.
\end{proof}

\begin{remark}
Note that the norm $\|\uHT\|_{\Vtr}$ on the right-hand side of~\eqref{eq:finalerr} may be further bounded using $\uHT = \hIHT\calQ \uHT$, \eqref{eq:stabhIHT}, and Lemma~\ref{lem:infsup}, i.e.,
\begin{equation*}
\|\uHT\|_{\Vtr} \leq \hCint H^{-1} \|(\Id+\calQ)\uHT\|_{\Vtr} \leq c^{-1}_\frakA\Cint\hCint H^{-1}\, \|f\|_{L^2(H^{-1})}.
\end{equation*}
\end{remark}

\section{Implementation}
\label{s:implementation}
This section is devoted to implementation aspects of the proposed localized multiscale method as introduced in~\eqref{eq:loccoarseproblem} including some details on the solution of the localized corrector problems. 
In this section, we denote the coarse time discretization as before by $0=: T_0 < T_1 < \cdot \cdot \cdot < T_{\NT} := \frakT$, and on each coarse temporal interval $[T_{j-1},T_j]$ we introduce an internal finer discretization $t_i := T_{j-1} + i\,\tau,\,0 \leq i \leq \Nt$ with $t_{\Nt} := T_j$, with uniform fine time step $\tau$.

First, we show how the localized sequential functions from problem~\eqref{eq:xiK} are solved. We restrict the detailed explanation to the first coarse temporal interval and then comment on the computations of the remaining intervals. Let $K\in \calTH$ be a coarse element and $\Lambda\in\hVHT$ a basis function which does not vanish on $K \times (0,T_1)$. On the first coarse interval, we seek $\xiKf \in \hVhtkK$ with $\xiKf(0) = 0$ such that 
\begin{subequations}\label{eq:constraintProbOnFirstPatch}
\begin{align}
\int_0^{T_1} \la \tddt\xiKf,v \ra + a(\xiKf,v) + \la \lamKkf, \IH v \ra \dt &= -\int_0^{T_1} \la \dot\Lambda,v \ra_K + a(\Lambda,v)_K \dt, \label{eq:constraintCorProb1a} \\
\la\IH\xiKf(T_1),\mu\ra &= 0  \label{eq:constraintCorProb1b}
\end{align}
\end{subequations}
for all $v \in \VhtkK,\, \mu \in \VH$, where $\lamKkf \in \VHkK$ is the associated Lagrange multiplier. Note that the computations are performed on the spatial patch $\Nb^k(K)$.
	
For illustrative purposes, we only consider one element $K \in \calTH$ in the following and a fixed spatial localization parameter $k \in \N$.
Further, we emphasize that the choice of the nodal interpolation operator in time allows for sequential computations from one coarse time interval to another. In particular, we only require the value of $\xiKf$ at $T_1$ to compute the solution on $[T_1,T_2]$ etc.

Note that throughout this section, we use the same notation for discrete functions and corresponding vectors, and abbreviate $\xi = \xiKf$ and $\lambda = \lamKkf$. 
Let $\Mhk$ and $\MHk$ be the localized mass matrices corresponding to the patch $\Nb^k(K)$ and the spaces $\VhkK$ and $\VHkK$, respectively. Besides, $\IHk$ denotes the matrix representation of $\IH\id_{\Nb^k(K)}$ and $\{\Shki\}_{i=1}^{\Nt}$ are the localized stiffness matrices with the coefficient evaluated at times $t_{i-1/2}$, $i=1,\ldots,\Nt$.

The scheme~\eqref{eq:constraintCorProb1a} now reduces to seeking vectors $\lambda$ and $\{\xi^{i}\}_{i=1}^{\Nt}$ with $\xi^{0} = 0$ such that
\begin{align}\label{eq:CNmatrix}
(\Mhk + \tfrac{\tau}{2}\Shki)\,\xi^{i} = (\Mhk - \tfrac{\tau}{2}&\Shkim)\,\xi^{i-1} - \tau \IHk^\mathrm{T}\MHk\,\lambda + \mathcal{F}^i
\end{align}
for $i=1,2,\ldots, \Nt$, where $\mathcal{F}^i = \MhkK(\Lambda(t_i) - \Lambda(t_{i-1})) + \frac{\tau}{2}\,(\ShkiK\Lambda(t_i) + \ShkimK\Lambda(t_{i-1}))$, where $\MhkK$, $\ShkiK$, and $\ShkimK$ are the mass matrix and stiffness matrices corresponding to $\VhkK$ localized to the element $K$. The scheme can be rephrased to the matrix system 
\begin{align}
\begin{bmatrix}
A & B
\end{bmatrix}
\begin{bmatrix}
\xi \\ \lambda
\end{bmatrix} = \mathcal{F}, \label{eq:matrixsystembutnotcompleteformyet}
\end{align}
with $\xi = [
(\xi^{1})^\mathrm{T}, \ldots, (\xi^{\Nt})^\mathrm{T} ]^\mathrm{T}$, $B=[\IHk^\mathrm{T}\MHk,\ldots,\IHk^\mathrm{T}\MHk]^\mathrm{T}$, $\mathcal{F} = [(\mathcal{F}^1)^\mathrm{T},\ldots,(\mathcal{F}^{\Nt})^\mathrm{T}]^\mathrm{T}$, and the block matrix $A$ describes the Crank--Nicolson scheme with entries
\begin{align*}
A_{ij} = \begin{cases}
\Mhk + \tfrac{\tau}{2}\Shki, \ &\text{if $i=j$}, \\
-\Mhk + \tfrac{\tau}{2}\Shkim, \ &\text{if $i = j+1$}.
\end{cases}
\end{align*}
Note that the entries in the matrices and vectors defined above are themselves matrices and vectors and hence we have, e.g., $A\in \mathbb{R}^{\Nh \cdot \Nt \times \Nh \cdot \Nt}$, where $\Nh$ denotes the number of fine degrees of freedom on the patch $\Nb^k(K)$.

Since also condition~\eqref{eq:constraintCorProb1b} needs to be satisfied,  the system~\eqref{eq:matrixsystembutnotcompleteformyet} changes to
\begin{align}
\begin{bmatrix}
A & B \\
C & 0
\end{bmatrix}
\begin{bmatrix}
\xi \\ \lambda
\end{bmatrix} = \begin{bmatrix}
\mathcal{F} \\ 0
\end{bmatrix}, \label{eq:fullmatrixsystem}
\end{align}
where the matrix $C = [C_1,\ldots,C_{\Nt}]$ is such that $C_i = 0$ for $i=1,2,\ldots,\Nt-1$ and $C_{\Nt} = \IHk$. From now on, a row where each element itself is a matrix will be referred to as a \emph{block-row}. For instance, the matrix $C$ can be called a block-row. 

To solve~\eqref{eq:constraintProbOnFirstPatch}, we need to solve the matrix system~\eqref{eq:fullmatrixsystem}. The reason why we need to solve a linear system of block matrices is that we have a constraint at time $T_1$ in addition to the initial condition. If solved naively, this would mean that we need to solve a huge system if $\tau$ is small. However, this can be avoided by the procedure outlined below. In particular, 
we employ the Schur complement method (see, e.g., \cite[Sect.~1.1]{Zha05}) to this block matrix. 

Multiplying the first equation in~\eqref{eq:fullmatrixsystem} from the left side by $-CA^{-1}$, and using the condition $C\xi = 0$ from the second equation, we arrive at
\begin{align*}
CA^{-1}B\,\lambda = CA^{-1}\mathcal{F}.
\end{align*}
At this point, recall that $C_i=0$ for $i=1,\ldots,\Nt-1$ and only $C_{\Nt}=\IHk \neq 0$. Hence, it suffices to compute the last block-row of $A^{-1}B$ and apply the localized interpolant to it. This is equivalent to computing the solution to the system $AX = B$ and extracting the last block-row from $X$. Further, we note that $A$ solely describes the Crank--Nicolson scheme and we can therefore obtain $X$ sequentially, similar to~\eqref{eq:CNmatrix}. With $X$ computed, it holds that $CA^{-1}B = \IHk X_{\Nt}$. Likewise, we can for the right-hand side solve $AY = \mathcal{F}$ sequentially, and get $CA^{-1} = \IHk Y_{\Nt}$. We consequently find $\lambda$ as the solution to $\IHk X_{\Nt}\lambda= \IHk Y_{\Nt}$. With $\lambda$ computed, $\xi$ can be obtained as the solution to $A\xi= \mathcal{F} - B\lambda$, which can once again be solved sequentially with the Crank--Nicolson scheme. This gives the first contribution $\xiKf = \xi$. Note that in order to obtain $\xi$, $\NH + 2$ parabolic problems need to be solved in the interval $[0,T_1]$, where $\NH$ refers to the number of coarse degrees of freedom in the patch $\Nb^k(K)$. We emphasize that parallelization is possible.

This procedure is repeated on the intervals $[T_{j-1},T_j]$ for $j=2,\ldots,\ell$ with similar computations, but with different right-hand sides; cf.~\eqref{eq:xiK}. With the correctors for each coarse interval computed, we can construct the full localized corrector as described in~\eqref{eq:loclocCor}.
We summarize the main steps for the computation of the operator $\Qkl$ in Algorithm~\ref{alg:mainalgo}. We emphasize here that the for-loops over $K \in \calTH$ and $i = 1,\ldots,\NT$ can be computed in a parallel manner, since the corresponding correctors are completely independent of each other. This means that the proposed method is fully parallel in both space and time except for the (cheap) final coarse-scale computations.
	
\begin{algorithm}
Choose localization parameters $k,\ell$\;
\For{$K\in \calTH$}{
	Compute local matrices $\IHk$, $\MHk$, $\Mhk$ and $\{\Shki\}_{i=1}^{\Nt}$\;
	\For{$i=1,2,\ldots,\NT$}{
		\For{ 
				basis functions $\Lambda \in \hVHT$ s.t.~$\supp(\Lambda)\cap \DKi \neq \emptyset$}{
			\For{$j=i,i+1,\ldots,\min\{{i+\ell-1},\NT\}$}{
				Construct $A$, $B$, and $\mathcal{F}$ from \eqref{eq:matrixsystembutnotcompleteformyet}\;
				Solve $AX = B$ sequentially\;
				Solve $AY = \mathcal{F}$ sequentially\;
				Solve $\IHk X_{\Nt}\lambda= \IHk Y_{\Nt}$\;
				Solve $A\xi= \mathcal{F} - B\lambda$ sequentially\;
				Set $\xiK = \xi$\;
			}
		}
	}
}
Construct the operator $\Qkl$ by \eqref{eq:loclocCor}\;
\caption{Computation of the correction operator $\Qkl$.}
\label{alg:mainalgo}
\end{algorithm}
	
The final step to compute the solution to~\eqref{eq:loccoarseproblem} consists in assembling the corresponding coarse matrices with the coarse trial space $(\Id+\Qkl)\hVHT$ and the test space $\VHT$. Once computed, \eqref{eq:loccoarseproblem} can be solved as a sequential scheme that involves information on the $\ell+1$ previous coarse approximations (due to the temporal support of the correctors). This is an extremely fast scheme and can be used for multiple right-hand sides (cf.~Section~\ref{ss:rhs}) without recomputing correctors. 
	
\section{Numerical Examples}
\label{s:numericalexamples}
In this section, we present numerical examples that illustrate the performance of the proposed localized space-time multiscale method. For all our examples, we consider the domain \mbox{${\Omega} = [0,1]\times [0,1]$}. The values of the (scalar) coefficients $A$ used in the examples are generated randomly within the interval $[0.01, 0.1]$ and are piecewise constant on an underlying mesh on the scale $\varepsilon_x$ in space and $\varepsilon_t$ in time. For computational convenience, the diffusion is also periodic in time and the period length coincides with the time step $\T$. For all our numerical experiments, we set $h=\tau=2^{-7}$, $\varepsilon_x = \varepsilon_t = 2^{-5}$, and $T = 1.25$.

First, we provide an example that illustrates how the localization error and the corresponding error indicators decay exponentially in both spatial and temporal sense, which justifies the above localization procedure. We then investigate the performance of our localized multiscale method with an example that shows the convergence behavior of first order with respect to the coarse mesh size and the coarse time step. Last, we show an example where the method is used repeatedly for several different right-hand sides by computing the coarse matrices in~\eqref{eq:loccoarseproblem} once and reusing them effectively.

\subsection{Localization error and decay of basis correctors}
\label{ss:decayexample}
	
The first example illustrates how a basis corrector is affected by the spatial localization parameter $k$ and temporal localization parameter $\ell$, respectively. Here, we choose the coarse basis function $\Lambda \in \hVHT$ associated to the node $x=(0.5,0.5)$ and the time point $t=\T$, compute the corresponding non-localized basis corrector $\calQ\Lambda$, and compare it with either the corrector $\calQ_{k,\infty}\Lambda$ that is only localized in space or $\calQ_{\infty,\ell}\Lambda$ with localization solely in time. 

For this example, we use $H=\T = 2^{-3}$, and let $k$ and $\ell$ vary between $1$ and $8$. The errors are measured in the (relative) trial norm $\|\cdot\|_{\Vtr}$ and are plotted in Figure~\ref{fig:basisdecayerror} with respect to $k$ and $\ell$, respectively. Figure~\ref{fig:spacedecay} shows the exponential decay of the localization error $(\calQ-\calQ_{k,\infty})\Lambda$ with increasing $k$ and in Figure~\ref{fig:timedecay} the exponential decay of the error $(\calQ-\calQ_{\infty,\ell})\Lambda$ can be observed. We also present the decay of the error indicators $\delta_{k,\infty}$ and $\vartheta_{\infty,\ell}$, which show the same decay rates as the spatial and the temporal localization errors, respectively. This justifies the use of these indicators in order to determine whether localization parameters need to be adjusted. 
Since constants are neglected, some tuning will be needed if the indicators should be used as absolute bounds of the error.

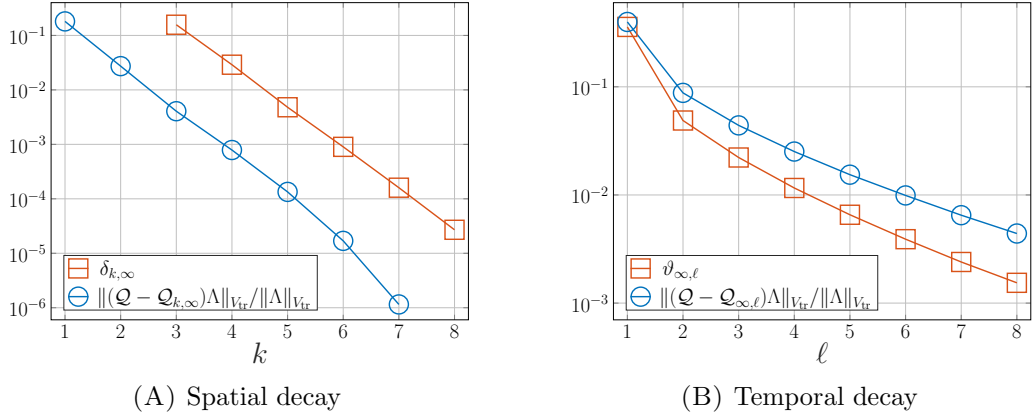
\begin{figure}
	\centering
	\begin{subfigure}[b]{0.49\textwidth}
		\definecolor{mycolor1}{rgb}{0.00000,0.44700,0.74100}
		\definecolor{mycolor2}{rgb}{0.85000,0.32500,0.0980}
		\definecolor{mycolor3}{rgb}{0.9290,0.6940,0.1250}
		\centering
		\begin{tikzpicture}[scale=0.45]
			\begin{axis}[%
				width=4.854in,
				height=3.692in,
				at={(0.814in,0.506in)},
				scale only axis,
				xmin=.75,
				xmax=8.25,
				xtick={1,2,3,4,5,6,7,8},
				xticklabels={\LARGE{$1$},\LARGE{$2$},\LARGE{$3$},\LARGE{$4$},\LARGE{$5$},\LARGE{$6$},\LARGE{$7$},\LARGE{$8$}},
				xlabel style={font=\color{white!15!black}},
				xlabel={\Huge $k$},
				ymode=log,
				ymin = 0.0000006,
				ymax=0.4,
				yminorticks=true,
				ytick={1,0.1,0.01,0.001,0.0001,0.00001,0.000001},
				yticklabels={ \LARGE{$\text{10}^{\text{0}}$},\LARGE{$\text{10}^{-\text{1}}$},\LARGE{$\text{10}^{-\text{2}}$},\LARGE{$\text{10}^{-\text{3}}$},\LARGE{$\text{10}^{-\text{4}}$},\LARGE{$\text{10}^{-\text{5}}$},\LARGE{$\text{10}^{-\text{6}}$}},
				axis background/.style={fill=white},
				xmajorgrids,
				ymajorgrids,
				legend style={legend cell align=left, align=left, draw=white!15!black},
				legend style={row sep=0.2cm},
				legend pos={south west}
				]
				\addplot [color=mycolor2, mark=square, mark options={solid, mycolor2},  very thick,mark size = 8]
				table[row sep=crcr]{%
					3	0.1572\\
					4	0.029\\
					5	0.0048\\
					6	0.0009\\
					7	0.00016\\
					8	0.000027\\
				};
				\addlegendentry{\LARGE \;\,$\delta_{k,\infty}$}		
				\addplot [color=mycolor1, mark=o, mark options={solid, mycolor1},  very thick,mark size = 8]
				table[row sep=crcr]{%
					1	0.181\\
					2	0.0273\\
					3	0.00406\\
					4	0.000787\\
					5	0.000134\\
					6	0.0000169\\
					7	0.00000115\\
				};
				\addlegendentry{\LARGE \;$\|(\calQ - \calQ_{k,\infty})\Lambda\|_{\Vtr}/\|\Lambda\|_{\Vtr}$}
			\end{axis}
		\end{tikzpicture}
		\caption{\small Spatial decay}
		\label{fig:spacedecay}
	\end{subfigure}
	~
	\begin{subfigure}[b]{0.49\textwidth}
		\centering
		\definecolor{mycolor1}{rgb}{0.00000,0.44700,0.74100}
		\definecolor{mycolor2}{rgb}{0.85000,0.32500,0.0980}
		\begin{tikzpicture}[scale=0.45]
			
			\begin{axis}[%
				width=4.854in,
				height=3.692in,
				at={(0.814in,0.506in)},
				scale only axis,
				xmin=.75,
				xmax=8.25,
				xtick={1,2,3,4,5,6,7,8},
				xticklabels={\LARGE{$1$},\LARGE{$2$},\LARGE{$3$},\LARGE{$4$},\LARGE{$5$},\LARGE{$6$},\LARGE{$7$},\LARGE{$8$}},
				xlabel style={font=\color{white!15!black}},
				xlabel={\Huge $\ell$},
				ymode=log,
				ymin=0.0007,
				ymax=.6,
				yminorticks=true,
				ytick={1,0.1,0.01,0.001,0.0001,0.00001,0.000001},
				yticklabels={ \LARGE{$\text{10}^{\text{0}}$},\LARGE{$\text{10}^{-\text{1}}$},\LARGE{$\text{10}^{-\text{2}}$},\LARGE{$\text{10}^{-\text{3}}$},\LARGE{$\text{10}^{-\text{4}}$},\LARGE{$\text{10}^{-\text{5}}$},\LARGE{$\text{10}^{-\text{6}}$}},
				axis background/.style={fill=white},
				xmajorgrids,
				ymajorgrids,
				legend style={legend cell align=left, align=left, draw=white!15!black},
				legend style={row sep=0.2cm},
				legend pos={south west}
				]
				\addplot [color=mycolor2, mark=square, mark options={solid, mycolor2}, very thick,mark size = 8]
				table[row sep=crcr]{%
					1	0.358\\
					2	0.0488\\
					3	0.0222\\
					4	0.0116\\
					5	0.00656\\
					6	0.00390\\
					7	0.00240\\
					8	0.00154\\
				};
				\addlegendentry{\LARGE \;\,$\vartheta_{\infty,\ell}$}
				\addplot [color=mycolor1, mark=o, mark options={solid, mycolor1}, very thick,mark size = 8]
				table[row sep=crcr]{%
					1	0.398\\
					2	0.0879\\
					3	0.0440\\
					4	0.0252\\
					5	0.0154\\
					6	0.0099\\
					7	0.0065\\
					8	0.0044\\
				};
				\addlegendentry{\LARGE \;$\|(\calQ - \calQ_{\infty,\ell})\Lambda\|_{\Vtr}/\|\Lambda\|_{\Vtr}$}
			\end{axis}
		\end{tikzpicture}
		\caption{\small Temporal decay}
		\label{fig:timedecay}
	\end{subfigure}
	\caption{\small Relative localization errors for a fixed basis function $\Lambda$ and values of the error indicators with respect to varying parameters $k$ or $\ell$.
	}\label{fig:basisdecayerror}
\end{figure}

Apart from the error curves, we also present in Figure~\ref{fig:basisdecay} an illustration of the basis corrector $\calQ\Lambda$ at the time points $T_j + \tau,\,j = 0,\ldots,3$. These time points are the first fine time steps within the first four coarse temporal intervals, respectively. Here, we clearly see a quick temporal decay after the first two intervals (on which the function $\Lambda$ is supported) and the spatial decay of the corrector as well.	

\subsection{Full method}
	
The second example illustrates how the proposed localized multiscale method from Section~\ref{s:locmethod} converges with first order in space and time. For this example, the coarse parameters $H = \T$ vary within $\{2^{-2}, 2^{-3}, \ldots, 2^{-6}\}$. We further set $k=|\log_2(H)|$ for the localization in space and $\ell=4$ for the localization in time. 

We compute the solution $\tuHTkl$ to~\eqref{eq:loccoarseproblem} with $f \equiv 1$ and compare it to a reference solution $\uht$, 
computed with~\eqref{eq:discretizedsol}, which resolves the oscillation scales. We measure the error in the relative trial norm, i.e, $\|\tuHTkl - \uht\|_{\Vtr}/\|\uht\|_{\Vtr}$, and plot it together with a reference line $\mathcal{O}(H)$ in Figure~\ref{fig:error_full_method} (left). One observes that the convergence rate of Theorem~\ref{t:errIdeal} is maintained even with the applied localization. For a comparison, we also present the errors in the relative $L^2(H^1_0)$-norm, which leads to very similar relative errors.

\begin{figure}
	\centering
	\begin{subfigure}[b]{0.44\textwidth}
		\includegraphics[width=\textwidth]{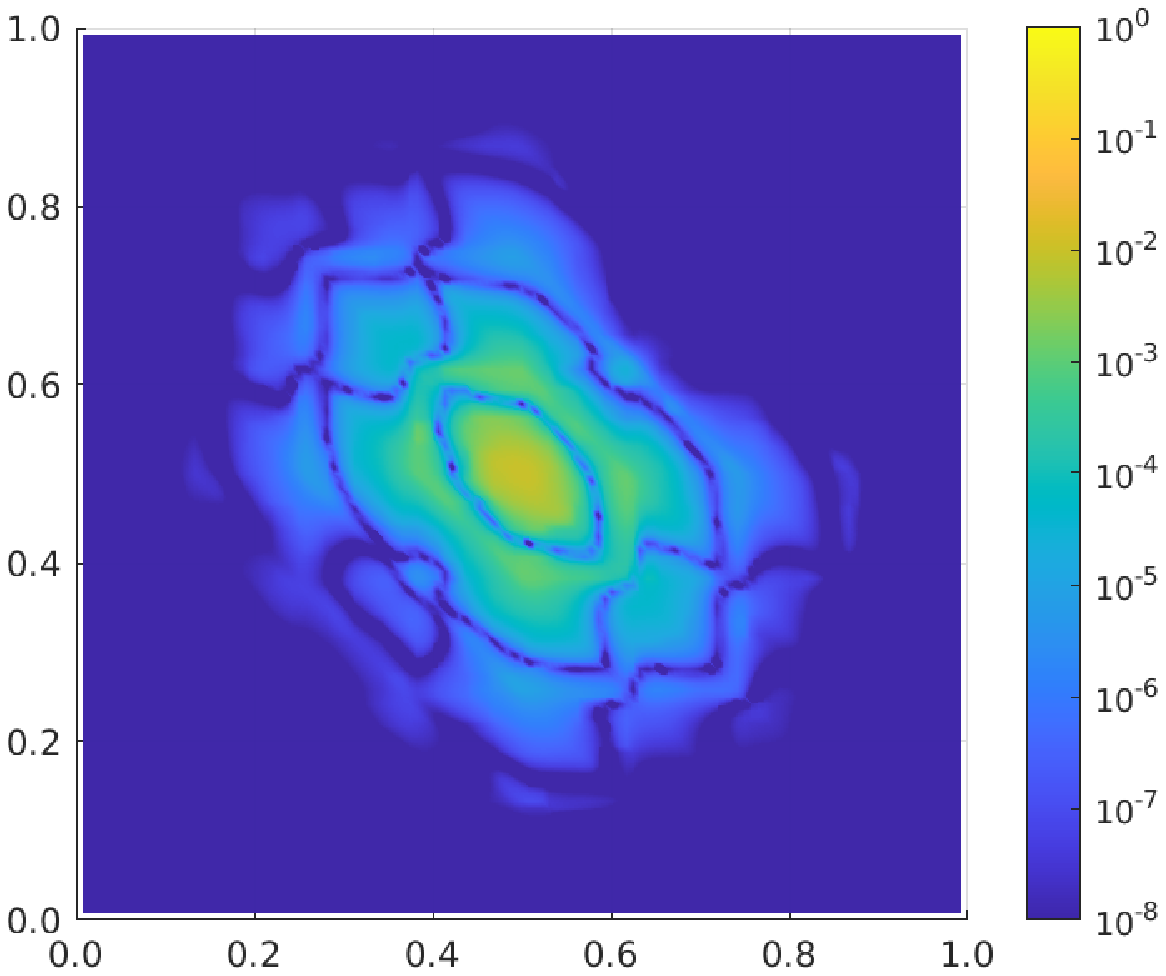}
		\caption{\small $\calQ\Lambda(\cdot,\tau)$}
	\end{subfigure}
	~ 
	\begin{subfigure}[b]{0.44\textwidth}
		\includegraphics[width=\textwidth]{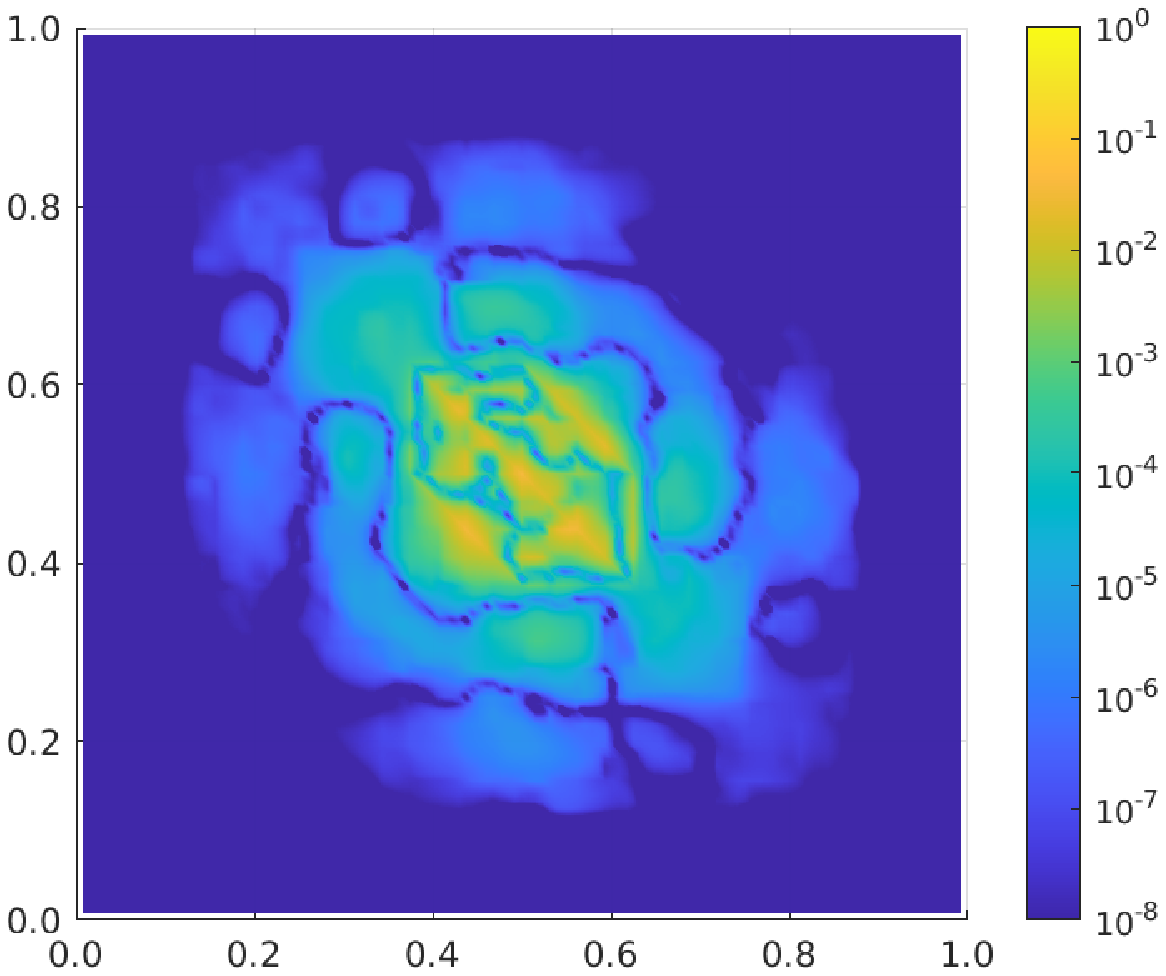}
		\caption{\small $\calQ\Lambda(\cdot,T_1+\tau)$}
	\end{subfigure} \\
	\begin{subfigure}[b]{0.44\textwidth}
		\includegraphics[width=\textwidth]{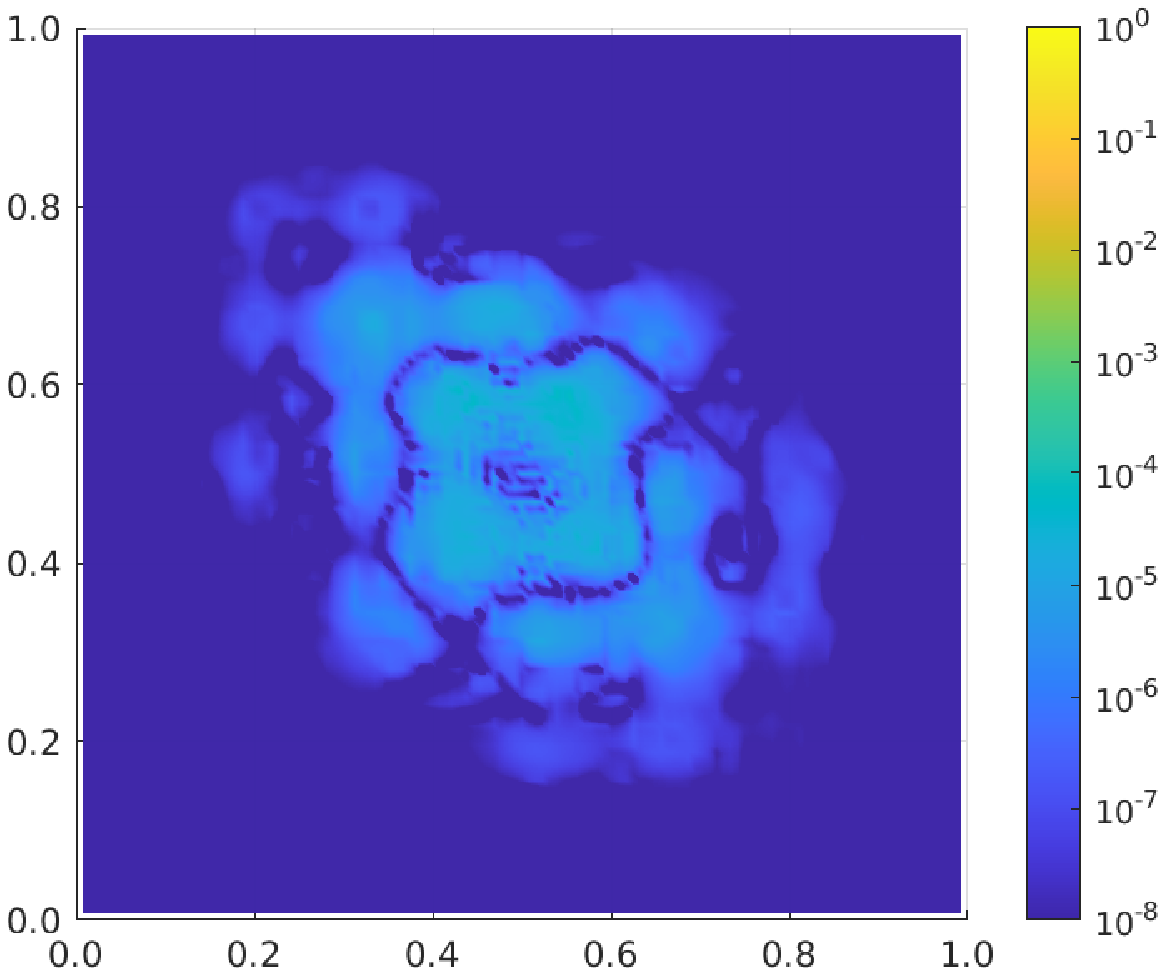}
		\caption{\small $\calQ\Lambda(\cdot,T_2+\tau)$}
	\end{subfigure}
	~ 
	\begin{subfigure}[b]{0.44\textwidth}
		\includegraphics[width=\textwidth]{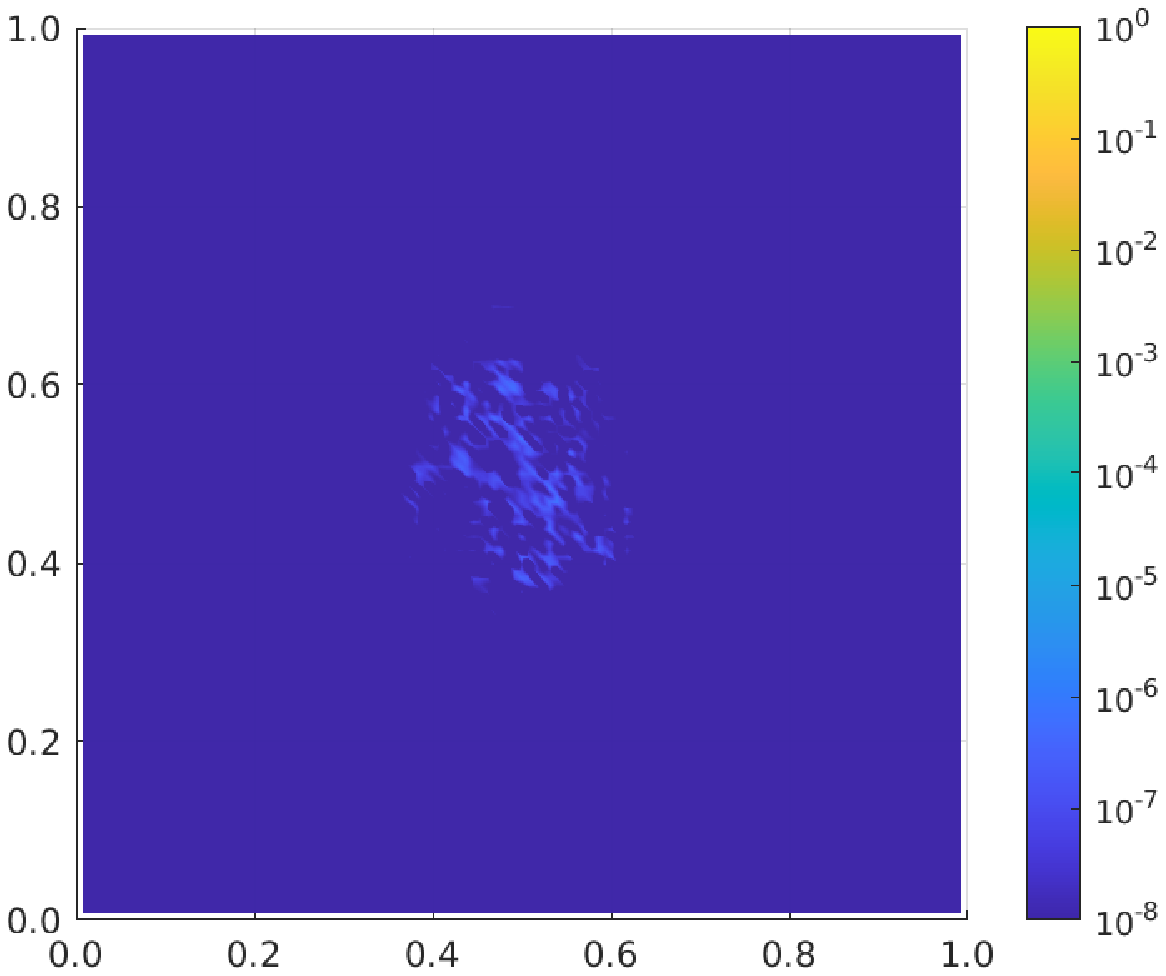}
		\caption{\small $\calQ\Lambda(\cdot,T_3+\tau)$}
	\end{subfigure}
	\caption{\small Illustration of the decay of a basis corrector $\calQ\Lambda$ in logarithmic scale.
	}\label{fig:basisdecay}
\end{figure}
	
\subsection{Multiple right-hand sides}\label{ss:rhs}

As a final example, we demonstrate how the method can be used to efficiently compute the solution to a new system in which the source function is different. We choose $H = \T = 2^{-4}$ and, as above, we set $k=|\log_2(H)|$ and $\ell=4$. 

To demonstrate the performance, we run the method for 8000 systems where the source function is randomized. In particular, it is chosen as
$f(t,x) = \tilde{f}(x) + a + bt + ct^2$, 
where $a,b,c$, and $\tilde{f}(x)$ are uniformly distributed random variables taking values from $0$ to $1$ for all nodes $x\in \mathcal{N}_h$ in the fine mesh. For each right-hand side, the relative error between the computed solution $\tuHTkl$ 
and the reference solution $\uht$  
is measured in the trial norm and stored. These 8000 errors are then plotted in a histogram, depicted in Figure~\ref{fig:error_full_method} (right). One observes that the values always fall in the range of $[0.0685,0.0845]$, which shows how the method can be consistently and reliably reused for new source functions. Once the coarse-scale representation is computed, for each right-hand side only $\NT$ matrix systems of size $\NH \times \NH$ need to be solved, where $\NH$ denotes the total number of degrees of freedom on the coarse scale. In contrast, using a combination of the Crank--Nicolson scheme combined with a classical finite element method (cf.~\eqref{eq:discretizedsol}) on a fine scale that resolves the oscillations in the coefficient would require $\Nt$ solutions of an $\Nh\times \Nh$ matrix system, where $\Nh$ and $\Nt$ are the fine global degrees of freedom in space and time, respectively. The method therefore provides a huge speed-up when multiple right-hand sides are considered.
	
\section{Conclusions}\label{s:conclusion}

In this work, we have presented and analyzed a space-time multiscale method for a parabolic model problem where the diffusion coefficient is highly oscillatory in both space and time. The proposed method is based on the framework of the Variational Multiscale Method and adopts ideas from the Localized Orthogonal Decomposition method. The approach computes a coarse-scale representation of the differential operator which appropriately incorporates fine-scale features by so-called corrections. These corrections are completely independent of each other and allow for parallel computations in space and in time. The coarse representation comes along with great approximation properties, even in the under-resolved scheme, where oscillations in the coefficient are not resolved, and allows us to efficiently compute approximations for multiple different right-hand sides. We have proved first-order convergence for an ideal method and illustrated the exponential decay of the corrections in both space and time, which motivated a localized version of the method that is computationally very efficient. We have showed a posteriori estimates for the localization error and presented an error estimate for our localized multiscale method. Finally, numerical examples have been provided that demonstrate the decay properties and the convergence of the proposed method. 

While the presented theoretical estimates and numerical considerations certainly show the potential of our approach, a remaining question is still whether reliable a priori localization estimates can be shown. This is addressed in future research.

\begin{figure}
	\begin{center}
		\scalebox{0.7}{
			\definecolor{mycolor1}{rgb}{0.00000,0.44700,0.74100}%
			\definecolor{mycolor2}{rgb}{0.85000,0.32500,0.09800}%
			\definecolor{mycolor3}{rgb}{0.9290,0.6940,0.1250}
			\begin{tikzpicture}
				
				\begin{axis}[%
					scale=0.72,
					width=4.854in,
					height=3.682in,
					at={(0.814in,0.516in)},
					scale only axis,
					xmin=1.85,
					xmax=6.14,
					xtick={2,3,4,5,6},
					xticklabels={\large{$\text{2}^{-\text{2}}$},\large{$\text{2}^{-\text{3}}$},\large{$\text{2}^{-\text{4}}$},\large{$\text{2}^{-\text{5}}$},\large{$\text{2}^{-\text{6}}$}},
					xlabel style={font=\color{white!15!black}},
					xlabel={\large $\T=H$},
					ymin=-7.5,
					ymax=-0.8,
					ytick={-7,-6,-5,-4,-3,-2,-1},
					yticklabels={\large{$\text{2}^{-\text{7}}$},\large{$\text{2}^{-\text{6}}$},\large{$\text{2}^{-\text{5}}$},\large{$\text{2}^{-\text{4}}$},\large{$\text{2}^{-\text{3}}$},\large{$\text{2}^{-\text{2}}$},\large{$\text{2}^{-\text{1}}$}},
					ylabel style={font=\color{white!15!black}},
					axis background/.style={fill=white},
					xmajorgrids,
					ymajorgrids,
					legend style={legend cell align=left, align=left, draw=white!15!black},
					legend style={row sep=0.2cm}
					]
					\addplot [color=mycolor1, mark=o, mark options={solid, mycolor1},
					very thick,mark size = 5]
					table[row sep=crcr]{%
						2	-1.2807\\
						3	-2.6135\\
						4	-3.8741\\
						5	-5.3452\\
						6	-7.1380\\
					};
					\addlegendentry{\Large \,$\|\cdot\|_{\Vtr}$}
					\addplot [color=mycolor3, mark=x, mark options={solid, mycolor2},
					very thick,mark size = 5]
					table[row sep=crcr]{%
						2	-1.2536\\
						3	-2.5960\\
						4	-3.7700\\
						5	-5.3047\\
						6	-7.0589\\
					};
					\addlegendentry{\Large\,$\|\cdot\|_{L^2(H^1_0)}$}
					\addplot [color=black, dotted,very thick]
					table[row sep=crcr]{%
						2	-2\\
						3	-3\\
						4	-4\\
						5	-5\\
						6	-6\\
					};
					\addlegendentry{\Large \,order 1}
				\end{axis}
			\end{tikzpicture}%
		}
		\hfill
		\scalebox{.63}{
			\begin{tikzpicture}
				\definecolor{mycolor1}{rgb}{0.00000,0.44700,0.74100}%
				\begin{axis}[%
					scale=0.83,
					scaled ticks=false, 
					tick label style={/pgf/number format/fixed},
					width=4.602in,
					height=3.506in,
					at={(0.772in,0.473in)},
					scale only axis,
					xmin=0.06775,
					xmax=0.08425,
					xlabel style={font=\color{white!15!black}},
					xlabel={\Large rel.~trial error},
					xticklabel style={/pgf/number format/precision=3},
					ymin=0,
					ymax=600,
					ylabel style={font=\color{white!15!black}},
					ylabel={\Large $\#$rhs},
					axis background/.style={fill=white}
					]
					\addplot[ybar interval, fill=mycolor1, fill opacity=0.6, draw=black, area legend] table[row sep=crcr] {%
						x	y\\
						0.0685	4\\
						0.069	19\\
						0.0695	50\\
						0.07	119\\
						0.0705	174\\
						0.071	229\\
						0.0715	271\\
						0.072	348\\
						0.0725	424\\
						0.073	451\\
						0.0735	482\\
						0.074	512\\
						0.0745	558\\
						0.075	542\\
						0.0755	548\\
						0.076	478\\
						0.0765	426\\
						0.077	383\\
						0.0775	331\\
						0.078	293\\
						0.0785	265\\
						0.079	253\\
						0.0795	179\\
						0.08	176\\
						0.0805	118\\
						0.081	109\\
						0.0815	98\\
						0.082	74\\
						0.0825	45\\
						0.083	30\\
						0.0835	9\\
						0.084	2\\
						0.0845	2\\
					};
				\end{axis}
			\end{tikzpicture}%
		}
	\end{center}
	\caption{\small Relative error with respect to different $H = \scalebox{.9}{$\mathcal{T}$}$ (left) and for fixed $H = \scalebox{.9}{$\mathcal{T}$} = 2^{-5}$ and 8000 source functions in a histogram (right).}
	\label{fig:error_full_method}
\end{figure}
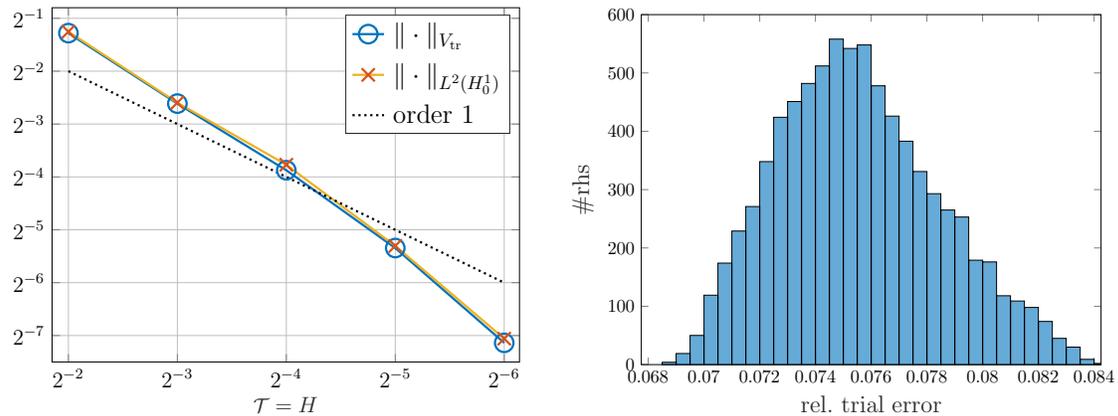

\subsection*{Acknowledgments} 
The authors acknowledge support by the G\"oran Gustafsson Foundation for Research in Natural Sciences and Medicine. The last author is also supported by the Swedish Research Council, project number 2019-03517\_VR.

\end{document}